\documentclass[a4paper,12pt]{amsart}
\usepackage[margin=2cm]{geometry}
\usepackage[alphabetic, initials]{amsrefs}

\usepackage{indentfirst}
\usepackage{fancyhdr}
\usepackage{amssymb}

\usepackage{pifont}

\usepackage{amsmath,mathtools} 
\mathtoolsset{showonlyrefs}
\usepackage{latexsym}
\usepackage{amsthm} 
\usepackage{eucal} 
\usepackage{enumerate} 
\usepackage{pdfpages}
\usepackage{mathrsfs}
\usepackage{stackrel}
\usepackage{enumerate}
\usepackage{graphicx}

\numberwithin{equation}{section}
\theoremstyle{plain}

\usepackage[T1]{fontenc}
\usepackage{aurical}
\usepackage{wesa}

\newtheorem{theorem}{Theorem}[section]
\newtheorem{lem}[theorem]{Lemma}
\newtheorem{cor}[theorem]{Corollary}
\newtheorem{prop}[theorem]{Proposition}
\theoremstyle{definition}

\newcommand{\R}{\mathbb{R}}
\newcommand{\N}{\mathbb{N}}

\renewcommand{\leq}{\leqslant}
\renewcommand{\le}{\leqslant}
\renewcommand{\geq}{\geqslant}
\renewcommand{\ge}{\geqslant}
\renewcommand{\epsilon}{\varepsilon}
\allowdisplaybreaks

\begin{document}

\title[Perturbative elliptic problems in the fractional setting]{Global perturbative elliptic problems \\ with critical growth \\ in the fractional setting}

\author[S. Dipierro, E. Proietti Lippi, E. Valdinoci]{Serena Dipierro,
Edoardo Proietti Lippi, Enrico Valdinoci}
\address{Department of Mathematics and Statistics,
The University of Western Australia, 35 Stirling Highway, Crawley, Perth, WA 6009, Australia.}
\email{serena.dipierro@uwa.edu.au}
\email{edoardo.proiettilippi@uwa.edu.au}
\email{enrico.valdinoci@uwa.edu.au}

\thanks{This work has been supported by the 
Australian Research Council
Future Fellowship FT230100333
{\em New perspectives on nonlocal equations}
and by the
Australian Laureate Fellowship FL190100081 {\em Minimal surfaces, free boundaries and partial differential equations}.}



\begin{abstract}
Given~$s$, $q\in(0,1)$, and a bounded and integrable function~$h$ which is strictly positive in an open set, we show that there exist at least two nonnegative solutions~$u$ of the critical problem
$$(-\Delta)^s u=\varepsilon h(x)u^q+u^{2^*_s-1},$$
as long as~$\varepsilon>0$ is sufficiently small.

Also, if~$h$ is nonnegative, these solutions are strictly positive.

The case~$s=1$ was established in~\cite{MR1801341},
which highlighted, in the classical case, the importance of combining perturbative techniques with variational methods: indeed, one of the two solutions branches off perturbatively in~$\varepsilon$ from~$u=0$, while the second solution is found by means of the Mountain Pass
Theorem.

The case~$s\in\left(0,\frac12\right]$ was already established, with different methods, in~\cite{MR3617721} (actually, in~\cite{MR3617721} it was erroneously believed that
the method would have carried through all the fractional cases~$s\in(0,1)$,
so, in a sense, the results presented here correct and complete the ones in~\cite{MR3617721}).
\end{abstract}

\maketitle

\bigskip\bigskip

\begin{center}
\begin{minipage}{30em}
{\wesa{{\footnotesize{$\langle\langle$}}\large{Scientists are very quick to say that science is self-correcting, but those who do the work behind this correction often
get accused of damaging their field, or worse.

Yes, error detectors can make research less comfortable
-- but that discomfort is healthy.}{\footnotesize{$\rangle\rangle$}}

(Simine Vazire, {\em Nature}, 2020)}}
\end{minipage}\end{center}
\bigskip\bigskip

\section{Introduction}

A classical topic in the field of nonlinear analysis is the study of
problems with critical growth. The literature on this topic is vast, but, in a sense, the common denominator of this line of research is that the critical power is typically the one for which, on the one hand, the standard methods of functional analysis become invalid, but, on the other hand, being in a borderline situation, several new phenomena may take place.

In this spirit, another source of criticality may arise when the problem is set in the whole of~$\R^N$, since mass could escape to infinity and standard compactness arguments may not be used in this situation.

Overcoming these difficulties required the introduction of powerful tools, which have now become foundational and ubiquitously used: in particular, the search of new solutions often relied on topological methods, such as the Mountain Pass Theorem, see e.g.~\cite{MR370183, MR1030853}, and the loss of compactness due to an unbounded domain was on many occasion compensated by the concentration-compactness principle,
see e.g.~\cite{MR653747}.

To construct new solutions, perturbation methods came in handy. In particular,
in~\cite{MR1801341}, global solutions of the problem
$$-\Delta u=\varepsilon h(x)u^q+u^{2^*-1},$$
where, as usual, $2^*:=\frac{2N}{N-2}$,
were obtained by an apt combination of the above methods:
specifically, $\varepsilon$ was considered as a small, positive parameter,
modulating a suitable sublinear perturbation, giving rise to two new solutions,
one branching off from the trivial one, and one created by the mountain pass structure of the perturbed functional.

The analogue of this problem in the fractional setting was first studied in~\cite{MR3617721}
and also contributed to the investigation of the concentration-compactness principle in the fractional setting. In particular, in~\cite{MR3617721} the existence theory was carried over for equations driven by the fractional Laplacian
$$ (-\Delta)^su(x):=\int_{\R^N}\frac{u(x)-u(y)}{|x-y|^{N+2s}}\,dx.$$
Here above and in the rest of this paper, normalizing constants are omitted for the sake of simplicity and the integral is taken in the sense of its principal value, to compensate for singularities.

The existence theory in~\cite{MR3617721} is valid for all~$s\in\left(0,\frac12\right]$.
However, it was incorrectly claimed in~\cite{MR3617721} that the methods utilized there would establish the same result for all the fractional parameter range~$s\in(0,1)$.

Motivated by an intriguing question by {\sc Norihisa Ikoma}, to whom we are deeply indebted, we show here
that, on the one hand, the methods of~\cite{MR3617721} are not sufficient by themselves to cover the full fractional parameter range~$s\in(0,1)$, but, on the other hand, the result
in~\cite{MR3617721} does hold true for all~$s\in(0,1)$, though with a different proof that we present here (actually, this proof is valid for all~$s\in(0,1)$, without needing to distinguish any threshold induced by the somewhat special exponent~$s=\frac12$).

The mathematical details go as follows.
We look for positive solutions\footnote{The standard definition of~$\mathcal{D}^{s,2}(\R^N)$ will be recalled in Section~\ref{NOTA}.}~$u\in \mathcal{D}^{s,2}(\R^N)$ of the fractional elliptic equation 
\begin{equation}\label{problema}
(-\Delta)^s u=\varepsilon h(x)u^q+u^{2^*_s-1},
\end{equation}
where~$s$, $q\in (0,1)$, with~$N>2s$, and~$2^*_s:=\frac{2N}{N-2s}$ is the fractional
critical Sobolev exponent.

We stress that solutions of~\eqref{problema} are ``global'', in the sense that the problem is set in the whole of~$\R^N$.

We make the following assumptions on~$h$:
\begin{itemize}
\item[$(h_0)$]
$h\in L^1(\R^N)\cap L^\infty(\R^N)$ ;
\item[$(h_1)$]
there exists a ball~$B\subset \R^N$ such that~$\displaystyle{\inf_B h>0}$.
\end{itemize}

The main result of this paper is thus the following:

\begin{theorem}\label{thduesoluzioni}
Let~$q\in (0,1)$. Suppose that~$h$ satisfies~$(h_0)$ and~$(h_1)$.

Then, there exists~$\varepsilon_0>0$ such that for all~$\varepsilon \in (0,\varepsilon_0)$ problem~\eqref{problema} has 
at least two nonnegative solutions. 

Moreover, these solutions belong to~$L^\infty(\R^N)\cap C^\alpha(\R^N)$, for any~$\alpha\in(0,\min\{2s,1\})$.

Also, if~$h\ge0$ in some open set~$\Omega\subseteq\R^N$, then these solutions are strictly positive in~$\Omega$.

In particular, if~$h\geq 0$,
then these solutions are strictly positive. 
\end{theorem}

We stress that, compared to the local case~$s=1$, the fractional case
presents further difficulties. To avoid such complications,
in~\cite{MR3617721} the Caffarelli-Silvestre extension method is used
to reduce the nonlocal problem~\eqref{problema} to a local one.
However, this strategy does not work for~$s\in (\frac{1}{2},1]$,
as we show in Appendix~\ref{APP}. Hence, here we deal directly with the
nonlocal problem, adapting the variational approach for the local
case in~\cite{MR1801341}.

With respect to~\cite{MR1801341}, however, several technical modifications are needed to address the nonlocal scenario.
The first main difference is that the concentration-compactness 
principle by Lions cannot be directly used to deal with the fractional case.
To overcome this complication, we adapt a concentration-compactness 
result in~\cite{MR3866572}. In particular, due to the nonlocality
of the probelm, using cut-off functions to prove convergence becomes more 
challenging as we have to deal with integrals presenting also contributions on 
unbounded domains. To resolve this issue we give some convergence 
results carefully tailored to address our framework.
We prove this by using a dyadic argument (see Lemmata~\ref{lemmalimiteDphidelta} and~\ref{lemmalimiteDphiR}).

Moreover, in order to prove the existence of a second solution,
in~\cite{MR1801341} the same localization procedure to find the first one is applied to a translated functional, defined by means of 
the positive part of a function (see~\cite[Proposition~2.5]{MR1801341}).
Here, the approach is not the same as in~\cite{MR1801341} and requires carefulness,
noting that the fractional gradient of the positive part of a 
function is harder to deal than the case of the classical
gradient (see Proposition~\ref{PSIepsilon}).\medskip

The rest of this paper is organized as follows.
In Section~\ref{NOTA}, for the sake of completeness, we recall some of the standard notation that will be used throughout this paper.
Then, in Section~\ref{PREL} we collect some auxiliary results available in the existing literature (with the modifications needed to make them effective in our setting).

The proof of Theorem~\ref{thduesoluzioni} is given in Sections~\ref{PRO1} and~\ref{PRO2}.
More specifically, Section~\ref{PRO1} shows the existence of a first solution as a local minimum,
while Section~\ref{PRO2} constructs a second solution via topological arguments.

We remark that the proofs carried out here will make no use of extension methods, and rather rely on suitable modifications of the original arguments in~\cite{MR1801341}, combined with the concentration-compactness theory put forth in~\cite{MR3866572}, as well as some integral estimates performed in~\cite{MR3271254}.

Finally, in Appendix~\ref{APP} we retake Proposition 3.1.1 in~\cite{MR3617721} and we show by a simple example that such a result cannot hold true when~$s>\frac12$,
thus highlighting the need of a different approach, as presented in this paper.

\section{Notation}\label{NOTA}

As customary, for all~$s\in(0,1)$, we consider the Gagliardo-Dirichlet form, defined by
\[
[u]_s^2:=\iint_{\R^{2N}}\frac{|u(x)-u(y)|^2}{|x-y|^{N+2s}}\,dx\,dy.
\]
and the functional space
\[
\mathcal{D}^{s,2}(\R^N):=
\Big\{u\in L^{2^*_s}(\R^N):\, [u]_s<+\infty \Big\}.
\]

We remark that
$$ [u]_s^2=\int_{\R^N}|D^su(x)|^2\,dx,$$
where
\[
|D^su(x)|^2:=\int_{\R^N}\frac{|u(x+h)-u(x)|^2}{|h|^{N+2s}}\,dh.
\]

We also consider the fractional Sobolev constant
\begin{equation}\label{definizioneS}
S:=\inf_{u\in \mathcal{D}^{s,2}(\R^N)}
\frac{[u]_s^2}{\|u\|_{L^{2^*_s}(\R^N)}^2}.
\end{equation}
\medskip

We will seek nonnegative solutions of~\eqref{problema} as 
critical points of the functional~$f_\varepsilon:\mathcal{D}^{s,2}(\R^N)\to \R$ defined as
\begin{equation}\label{deffepsilon}
f_\varepsilon(u):=
\frac{1}{2}\iint_{\R^{2N}}\frac{|u(x)-u(y)|^2}{|x-y|^{N+2s}}\,dx\,dy
-\frac{\varepsilon}{q+1}\int_{\R^N}h(x)u_+^{q+1}(x)\,dx
-\frac{1}{2^*_s}\int_{\R^N}u_+^{2^*_s}(x)\,dx.
\end{equation}
Notice indeed that, for all~$v\in\mathcal{D}^{s,2}(\R^N)$,
\begin{equation}\label{selanumero0}
\begin{split}
\langle f'_\varepsilon(u),v\rangle &=
\iint_{\R^{2N}}\frac{(u(x)-u(y))(v(x)-v(y))}{|x-y|^{N+2s}}\,dx\,dy\\
&\qquad-\varepsilon \int_{\R^N}h(x)u_+^{q}(x)v(x)\,dx
-\int_{\R^N}u_+^{2^*_s-1}(x)v(x)\,dx.
\end{split}\end{equation}

Moreover, critical points of the functional~\eqref{deffepsilon}
are nonnegative, according to Proposition~2.2.3 in~\cite{MR3617721}.

\section{Preliminaries from the literature}\label{PREL}

We gather here some useful results from the existing literature, stated in a form which is convenient for our purposes. For this, we let~$\phi \in W^{1,\infty}(\R^n,[0,1])$ be such that 
$\mbox{supp}(\phi)\subset B_1$ and~$\phi=1$ in~$B_{1/2}$. 

Given~$r>0$ and~$x_0\in \R^N$,
we define
\begin{equation}\label{definizionephi}
\phi_{r,x_0}(x):=\phi \left(\frac{x-x_0}{r}\right).
\end{equation}
In this setting, we recall Corollary~2.3 in~\cite{MR3866572}, used here with~$p=2$.

\begin{lem}\label{corollary2.3}
Let $\phi_{r,x_0}$ be defined as in~\eqref{definizionephi}.
Then, for all~$x\in\R^N$,
\[
|D^s\phi_{r,x_0}(x)|^2\leq C\min 
\left\{r^{-2s},r^N|x-x_0|^{-(N+2s)} \right\},
\]
where~$C>0$ depends on~$N$, $s$ and~$\|\phi_{r,x_0}\|_{W^{1,\infty}(\R^N)}$.
\end{lem}

Thee next result corresponds to Lemma~2.4 in~\cite{MR3866572}
with~$p=q=2$.

\begin{lem}\label{lemma2.4}
Let $w\in L^\infty(\R^N)$ and assume that there exist~$\alpha>0$ and~$C>0$ such that
\begin{equation}\label{fhewuoguo46859432}
0\leq w(x) \leq C|x|^{-\alpha}.
\end{equation}
Then, if~$\alpha> 2s$, the space~$\mathcal{D}^{s,2}(\R^N)$ is 
compactly embedded in~$L^2(w\,dx;\R^N)$, that is, for any sequence~$\{u_k\}_k\subset \mathcal{D}^{s,2}(\R^N)$ such that~$u_k \rightharpoonup u$ weakly in~$\mathcal{D}^{s,2}(\R^N)$ as~$k\to+\infty$,
we have that
\[
\lim_{k\to +\infty}
\int_{\R^N} |u_k(x)-u(x)|^2w(x)\,dx=0.
\]
\end{lem}


The proofs of the next two results can be found in~\cite{MR3866572}, but since they are not presented as stand alone
statements, we provide here the proof for the facility of the reader.
For the sake of completeness, we also mention that the argument after formula~(2.9)
in~\cite{MR3866572} should be amended by using a dyadic argument, as detailed
in the forthcoming Lemma~\ref{lemmalimiteDphiR}.

\begin{lem}\label{lemmalimiteDphidelta}
Let $u \in \mathcal{D}^{s,2}(\R^N)$ and~$\phi_{\delta,x_0}$ be defined 
as in~\eqref{definizionephi}. Then,
\[
\lim_{\delta \to 0}
\int_{\R^N}|u(x)|^2|D^s\phi_{\delta,x_0}(x)|^2 \,dx=0.
\]
\end{lem}

\begin{proof}
In light of Lemma~\ref{corollary2.3}, we see that
\begin{equation}\label{31BIS0}\begin{split}&
\int_{\R^N}|u(x)|^2|D^s\phi_{\delta,x_0}(x)|^2 \,dx\\&\qquad
\leq C\left(\delta^{-2s}\int_{B_\delta(x_0)}|u(x)|^2\,dx
+\delta^N\int_{\R^N\setminus B_\delta(x_0)}\frac{|u(x)|^2}{|x-x_0|^{N+2s}}\,dx \right).\end{split}
\end{equation}

Also, using the H\"older inequality with exponents~$\frac{2^*_s}2$
and~$\frac{N}{2s}$,
\begin{eqnarray*}
\delta^{-2s}\int_{B_\delta(x_0)}|u(x)|^2\,dx
\leq \delta^{-2s}\left(\int_{B_\delta(x_0)}|u(x)|^{2^*_s}\,dx \right)
^\frac{2}{2^*_s}|B_\delta|^{\frac{2s}{N}}
=C\left(\int_{B_\delta(x_0)}|u(x)|^{2^*_s}\,dx\right)^\frac{2}{2^*_s},
\end{eqnarray*}
for some~$C>0$ depending on~$N$ and~$s$ but not on~$\delta$.

Therefore, since~$u\in L^{2^*_s}(\R^N)$, it follows that
\begin{equation}\label{31BIS}
\lim_{\delta \to 0}\delta^{-2s}\int_{B_\delta(x_0)}|u(x)|^2\,dx=0.
\end{equation}

Furthermore, we point out that
\begin{equation}\label{LFSD}
\begin{aligned}
\delta^N\int_{\R^N\setminus B_\delta(x_0)}\frac{|u(x)|^2}{|x-x_0|^{N+2s}}\,dx
&=\sum_{k=0}^{+\infty} \delta^N \int_{B_{2^{k+1}\delta}(x_0)\setminus
B_{2^k\delta}(x_0)}\frac{|u(x)|^2}{|x-x_0|^{N+2s}}\,dx \\
&\leq \sum_{k=0}^{+\infty} \frac{1}{2^{k(N+2s)}}\frac{1}{\delta^{2s}}
\int_{B_{2^{k+1}\delta}(x_0)}|u(x)|^2\,dx \\
&\leq \sum_{k=0}^{+\infty} \frac{1}{2^{k(N+2s)}}\frac{1}{\delta^{2s}}
\left(\int_{B_{2^{k+1}\delta}(x_0)}|u(x)|^{2^*_s}\,dx \right)^\frac{2}{2^*_s}|B_{2^{k+1}\delta}|^\frac{2s}{N} \\
&=c \sum_{k=0}^{+\infty} \frac{1}{2^{Nk}}
\left(\int_{B_{2^{k+1}\delta}(x_0)}|u(x)|^{2^*_s}\,dx \right)^\frac{2}{2^*_s},
\end{aligned}
\end{equation}
where~$c$ only depends on~$N$ and~$s$.

Now, let~$\varepsilon>0$ and take~$k_0\in \N$, depending on~$\varepsilon$, such that
\[
c \sum_{k=k_0+1}^{+\infty} \frac{1}{2^{Nk}}<\varepsilon.
\]
In this way, we deduce from~\eqref{LFSD} that
\[
\begin{aligned}
\delta^N\int_{\R^N\setminus B_\delta(x_0)}\frac{|u(x)|^2}{|x-x_0|^{N+2s}}\,dx
&\leq  \varepsilon\|u\|_{L^{2^*_s}(\R^N)}^2
+c\sum_{k=0}^{k_0}\frac{1}{2^{Nk}}
\left(\int_{B_{2^{k_0+1}\delta}(x_0)}|u(x)|^{2^*_s}\,dx \right)^\frac{2}{2^*_s} \\
&=\varepsilon\|u\|_{L^{2^*_s}(\R^N)}^2
+C\left(\int_{B_{2^{k_0+1}\delta}(x_0)}|u(x)|^{2^*_s}\,dx \right)^\frac{2}{2^*_s},
\end{aligned}
\]
for some~$C>0$ depending on~$N$, $s$ and~$k_0$ (and therefore on~$\varepsilon$).

As a result, we obtain that
\[
\limsup_{\delta \to 0}
\delta^N\int_{\R^N\setminus B_\delta(x_0)}\frac{|u(x)|^2}{|x-x_0|^{N+2s}}\,dx
\leq \varepsilon\|u\|_{L^{2^*_s}(\R^N)}^2.
\]
Taking the limit as~$\varepsilon \to 0$,
we conclude that
\[
\lim_{\delta \to 0}
\delta^N\int_{\R^N\setminus B_\delta(x_0)}\frac{|u(x)|^2}{|x-x_0|^{N+2s}}\,dx
=0.
\]
{F}rom this, \eqref{31BIS0} and~\eqref{31BIS}, we obtain the desired result.
\end{proof}

\begin{lem}\label{lemmalimiteDphiR}
Let $u \in \mathcal{D}^{s,2}(\R^N)$ and~$\phi_{R,0}$ be defined as in~\eqref{definizionephi}.
Let also~$\phi_R:=1-\phi_{R,0}$.

Then,
\begin{equation*}
\lim_{R \to +\infty}
\int_{\R^N}|u(x)|^2|D^s\phi_R(x)|^2 \,dx=0.
\end{equation*}
\end{lem}

\begin{proof} 
By Lemma~\ref{corollary2.3}, we have that
\begin{equation}\label{78PIT.09}
|D^s\phi_{R}(x)|^2=|D^s\phi_{R,0}(x)|^2\leq C\min 
\left\{R^{-2s},R^N|x|^{-(N+2s)} \right\}.
\end{equation}
As a consequence,
\begin{equation}
\label{dweioty3480y6ty24t23yghevi0}
\int_{\R^N}|u(x)|^2|D^s\phi_R(x)|^2 \,dx
\le C\left(R^{-2s}\int_{B_R}|u(x)|^2\,dx
+R^N\int_{\R^N\setminus B_R}\frac{|u(x)|^2}{|x|^{N+2s}}\,dx\right).
\end{equation}

Now, we observe that
\begin{equation*}\begin{split}
R^{-2s}\int_{B_R}|u(x)|^2\,dx&=R^{-2s}\sum_{k=0}^{+\infty}\int_{B_{{R}/{2^k}}\setminus  B_{{R}/{2^{k+1}}}}|u(x)|^2\,dx\\&\le
R^{-2s}\sum_{k=0}^{+\infty}\left(
\int_{B_{{R}/{2^k}}\setminus  B_{{R}/{2^{k+1}}}}|u(x)|^{2^*_s}\,dx\right)^{\frac2{2^*_s}}
|B_{{R}/{2^k}}|^{\frac{2s}N}\\&\le
c\sum_{k=0}^{+\infty}\frac1{2^{2sk}}\left(
\int_{B_{{R}/{2^k}}\setminus  B_{{R}/{2^{k+1}}}}|u(x)|^{2^*_s}\,dx\right)^{\frac2{2^*_s}}.
\end{split}
\end{equation*}

We point out that for all~$\varepsilon>0$ we can find~$k_0$ such that
$$ c\sum_{k=k_0+1}^{+\infty}\frac1{2^{2sk}}<\varepsilon,$$
and therefore
\begin{eqnarray*}
R^{-2s}\int_{B_R}|u(x)|^2\,dx &\le&\varepsilon\|u\|_{L^{2^*_s}(\R^N)}+
c\sum_{k=0}^{k_0}\frac1{2^{2sk}}\left(
\int_{B_{{R}/{2^k}}\setminus  B_{{R}/{2^{k+1}}}}|u(x)|^{2^*_s}\,dx\right)^{\frac2{2^*_s}}\\&\le&
\varepsilon\|u\|_{L^{2^*_s}(\R^N)}+
c\sum_{k=0}^{k_0}\frac1{2^{2sk}}\left(
\int_{\R^N\setminus  B_{{R}/{2^{k_0+1}}}}|u(x)|^{2^*_s}\,dx\right)^{\frac2{2^*_s}}.
\end{eqnarray*}
As a consequence,
$$ \limsup_{R\to+\infty} R^{-2s}\int_{B_R}|u(x)|^2\,dx\le \varepsilon\|u\|_{L^{2^*_s}(\R^N)}.$$
Sending~$\varepsilon\to0$, we obtain that
\begin{equation}\label{dweioty3480y6ty24t23yghevi}
 \lim_{R\to+\infty} R^{-2s}\int_{B_R}|u(x)|^2\,dx=0.
\end{equation}

Furthermore,
\begin{eqnarray*}
R^N\int_{\R^N\setminus B_R}\frac{|u(x)|^2}{|x|^{N+2s}}\,dx
&\le& R^N\left(\int_{\R^N\setminus B_R}|u(x)|^{2^*_s}\,dx\right)^{\frac2{2^*_s}}
\left( \int_{\R^N\setminus B_R}\frac{dx}{|x|^{\frac{N(N+2s)}{2s}}}\right)^{\frac{2s}N}\\
&=& C \left(\int_{\R^N\setminus B_R}|u(x)|^{2^*_s}\,dx\right)^{\frac2{2^*_s}},
\end{eqnarray*}
for some~$C>0$ depending on~$N$ and~$s$.

Since~$u\in L^{2^*_s}(\R^N)$, this gives that
$$\lim_{R\to+\infty}R^N\int_{\R^N\setminus B_R}\frac{|u(x)|^2}{|x|^{N+2s}}\,dx=0.$$
{F}rom this, \eqref{dweioty3480y6ty24t23yghevi0}
and~\eqref{dweioty3480y6ty24t23yghevi}, we obtain the desired claim.
\end{proof}

The following estimates corresponds to Lemmata~6.5.4, 6.5.5 and~6.5.6 in~\cite{MR3617721}, and we recall this statements for the facility of the reader.

\begin{lem}\label{lemma6.5.4}
For any~$a$, $b\geq 0$ and any~$p>1$, we have that
\[
(a+b)^p\geq a^p+b^p.
\]
Also, if $a$, $b>0$, we have that
\[
(a+b)^p> a^p+b^p.
\]
\end{lem}

\begin{lem}\label{lemma6.5.5}
Let $p\geq 2$. Then, there exists~$c_p>0$ such that, for any~$a$, $b\geq 0$,
\[
(a+b)^p\geq a^p+b^p+c_p a^{p-1}b.
\]
\end{lem}

\begin{lem}\label{lemma6.5.6}
Let $p\in (1,2)$ and $k>0$. Then, there exists~$c_{p,k}>0$ such 
that, for any~$a>0$ and~$b\geq 0$ with~$\frac{b}{a}\in [0,k]$,
\[
(a+b)^p\geq a^p+b^p+c_{p,k} a^{p-1}b.
\]
\end{lem}

\section{Existence of a local minimum}\label{PRO1}

In this section we prove the existence of a first solution to
problem~\eqref{problema} as a local minimim of the functional~$f_\varepsilon$ given in~\eqref{deffepsilon}. To this end,
we start with a result about the convergence of sequences of
functions in bounded domains.

\begin{lem}\label{lemmmaconvergenzaLrloc}
Let $K$ be any bounded set in~$\R^N$. Let~$\{u_n\}_n$ be a 
sequence in~$L^{2^*_s}(\R^N)$ converging, as~$n\to+\infty$, to some~$u$ in~$L^r_{loc}(\R^N)$ for some~$r\in [1,2^*_s)$.

Then,
\begin{equation}\label{eqLrloc1}
\lim_{n\to+\infty}\int_K |u_n^q(x)-u^q(x)|^\frac{r}{q}\,dx=0
\end{equation}
and
\begin{equation}\label{eqLrloc2}
\lim_{n\to+\infty}\int_K 
|u_n^{2^*_s-1}(x)-u^{2^*_s-1}(x)|^\frac{r}{2^*_s-1}\,dx=0.
\end{equation}
\end{lem}

\begin{proof}
For any~$t\geq -1$, let
\[
f(t):=\frac{|(1+t)^q-1|}{|t|^q}.
\]
Recalling that~$q\in(0,1)$, one can show that
\[
f(-1)=1,\qquad \lim_{t\to 0}f(t)=0 \qquad \mbox{and} \qquad
\lim_{t\to +\infty}f(t)=1.
\]
Thus, defining
\[
L:=\sup_{t\geq -1}f(t),
\]
we have that~$L\in[1,+\infty)$. 

Now we claim that
\begin{equation}\label{eq4.1.3}
|a^q-b^q|\leq L|a-b|^q
\end{equation}
for any~$a$, $b\geq 0$. Without loss of generality, we can suppose 
that~$b\neq 0$, and write~$t:=\frac{a}{b}-1$. 

Then, we have that
\[
|a^q-b^q|=b^q|(1+t)^q-1|\leq Lb^q|t|^q=L|a-b|^q,
\]
which proves~\eqref{eq4.1.3}.

{F}rom~\eqref{eq4.1.3} and the convergence of~$u_n$, we find that
\[
\lim_{n\to+\infty}\int_K |u_n^q(x)-u^q(x)|^\frac{r}{q}\,dx
\leq L^{\frac{r}{q}}\lim_{n\to+\infty}\int_K |u_n(x)-u(x)|^r\,dx=0,
\]
which establishes~\eqref{eqLrloc1}.

Now we prove~\eqref{eqLrloc2}. To do this, we first notice that,
for any~$a\geq b\geq 0$,
\[
a^{2^*_s-1}-b^{2^*_s-1}=(2^*_s-1)\int_b^a t^{2^*_s-2}\,dt
\leq (2^*_s-1)a^{2^*_s-2}(a-b)
\leq (2^*_s-1)(a+b)^{2^*_s-2}(a-b).
\]
Actually, by possibly exchanging the roles of~$a$ and~$b$, we conclude that, 
for any~$a$, $b\geq 0$,
\[
|a^{2^*_s-1}-b^{2^*_s-1}|\leq (2^*_s-1)(a+b)^{2^*_s-2}|a-b|.
\]
As a consequence, for any~$a$, $b\geq 0$,
\[
|a^{2^*_s-1}-b^{2^*_s-1}|^\frac{r}{2^*_s-1}
\leq (2^*_s-1)^\frac{r}{2^*_s-1}(a+b)^{\frac{2^*_s-2}{2^*_s-1}r}
|a-b|^\frac{r}{2^*_s-1}.
\]

{F}rom this and the H\"older inequality with exponents~$2^*_s-1$ and~$\frac{2^*_s-1}{2^*_s-2}$, we obtain that
\[
\begin{aligned}
\int_K &|u_n^{2^*_s-1}(x)-u^{2^*_s-1}(x)|^\frac{r}{2^*_s-1}\,dx \\
&\leq (2^*_s-1)^\frac{r}{2^*_s-1} \int_K
(u_n(x)+u(x))^{\frac{2^*_s-2}{2^*_s-1}r}
(u_n(x)-u(x))^\frac{r}{2^*_s-1}\,dx \\
&\leq (2^*_s-1)^\frac{r}{2^*_s-1}
\left( \int_K(u_n(x)+u(x))^r\,dx \right)^\frac{2^*_s-2}{2^*_s-1}
\left( \int_K(u_n(x)-u(x))^r\,dx \right)^\frac{1}{2^*_s-1} \\
&\leq (2^*_s-1)^\frac{r}{2^*_s-1}
\|u_n+u\|_{L^r(K)}^\frac{r(2^*_s-2)}{2^*_s-1}
\|u_n-u\|_{L^r(K)}^{r(2^*_s-1)}.
\end{aligned}
\]
{F}rom the convergence of~$u_n$, we have that~$\|u_n+u\|_{L^r(K)} \leq \|u_n\|_{L^r(K)}+\|u\|_{L^r(K)}$
is bounded uniformly in~$n$, while
\[
\lim_{n\to +\infty}\|u_n-u\|_{L^r(K)}=0,
\]
therefore~\eqref{eqLrloc2} holds.
\end{proof}

With this convergence result, we can establish the following:

\begin{lem}
\label{wea}
Let~$\{u_n\}_n$ be a sequence in~$\mathcal{D}^{s,2}(\R^N)$ weakly
converging, as~$n\to+\infty$, to some~$u$ in~$\mathcal{D}^{s,2}(\R^N)$.

Suppose that
\begin{equation}\label{PS-ii00}
{\lim_{n\to +\infty}}\sup_{{v\in\mathcal{D}^{s,2}(\R^N)}\atop{[v]_s\le1}}\big|\langle f'_\varepsilon(u_n), v\rangle \big|= 0.
\end{equation}

Then, $u$ is a nonnegative weak solution of~\eqref{problema}.
\end{lem}

\begin{proof}
We point out that
\begin{equation}\label{covergenzadebole00}
u_n\rightarrow u \mbox{ in } L^r_{loc}({\R^N}) \;{\mbox{ for every~$r\in [1,2^*_s)$}}
\qquad \mbox{and}\qquad
u_n\to u \mbox{ a.e. in } \R^N.
\end{equation}

Now we let~$v\in C^\infty_c(\R^N)$ and
we use~\eqref{PS-ii00} to say that
\begin{equation}\label{servetutyepn090909}\begin{split}
0&=\lim_{n\to +\infty}\langle f'_\varepsilon(u_n), v \rangle \\
&=\lim_{n\to +\infty} \Bigg(
\iint_{\R^{2N}}\frac{(u_n(x)-u_n(y))(v(x)-v(y))}{|x-y|^{N+2s}}\,dx\,dy \\
&\qquad\qquad\qquad
-\varepsilon\int_{\R^N}h(x)(u_n)_+^{q}(x)v(x)\,dx
-\int_{\R^N}(u_n)_+^{2^*_s-1}(x)v(x)\,dx \Bigg).
\end{split}\end{equation}
Let also~$K$ be a ball containing supp$(v)$.
Thanks to the strong convergence in~\eqref{covergenzadebole00}, 
we can use the H\"older inequality and
Lemma~\ref{lemmmaconvergenzaLrloc} to see that,
for every~$r\in (2^*_s-1,2^*_s)$, 
\begin{equation}\label{AL:2}\begin{split}&
\lim_{n\to +\infty}\left|
\int_{\R^N}h(x)\Big((u_n)_+^{q}(x)-u_+^{q}(x)\Big)v(x)\,dx\right|
\\&\qquad\qquad\leq \lim_{n\to +\infty}C 
\left(\int_{K}|(u_n)_+^{q}(x)-u_+^{q}(x)|^\frac{r}{q}\,dx \right)^{\frac{q}{r}}=0\end{split}
\end{equation}
and
\begin{equation}\label{AL:3}\begin{split}&
\lim_{n\to +\infty}\left|\int_{\R^N}\Big((u_n)_+^{2^*_s-1}(x)-u_+^{2^*_s-1}(x)\Big)v(x)\,dx\right|\\&\qquad\qquad\le
\lim_{n\to +\infty} C\left(\int_{K}|(u_n)_+^{2^*_s-1}(x)-u_+^{2^*_s-1}(x)|^\frac{r}{2^*_s-1} \,dx\right)^{\frac{2^*_s-1}r}=0,
\end{split}\end{equation}
for some~$C>0$.

Moreover, setting
\[
\xi_n(x,y):=\frac{u_n(x)-u_n(y)}{|x-y|^{\frac{N}{2}+s}},
\] we see that
the sequence~$\{\xi_n\}_n$ is uniformly bounded in~$L^2(\R^{2N})$, and so 
there exists~$\xi \in L^2(\R^{2N})$ such that
\begin{equation}\label{diory867684576hgdsjg}
{\mbox{$\xi_n \rightharpoonup \xi$ weakly in~$L^2(\R^{2N})$ as~$n\to+\infty$.}}\end{equation}

Notice also that the function~$\frac{v(x)-v(y)}{|x-y|^{\frac{N}{2}+s}}$
belongs to~$L^2(\R^{2N})$.
As a consequence,
\begin{equation}\label{AL:1}
\lim_{n\to +\infty}\iint_{\R^{2N}}\xi_n(x,y)
\frac{v(x)-v(y)}{|x-y|^{\frac{N}{2}+s}}\,dx\,dy 
=\iint_{\R^{2N}}\xi(x,y)
\frac{v(x)-v(y)}{|x-y|^{\frac{N}{2}+s}}\,dx\,dy .
\end{equation}
Furthermore, since~$u_n\to u$ a.e. in~$\R^N$, we have that,
a.e. in~$\R^{2N}$,
\begin{equation}\label{sevvbhg5847y750945uhbuf}
\lim_{n\to +\infty}\xi_n(x,y) = \frac{u(x)-u(y)}{|x-y|^{\frac{N}{2}+s}}.
\end{equation}

Now we claim that
\begin{equation}\label{asdfghjkqwertyuio23456789}
\xi(x,y)=\frac{u(x)-u(y)}{|x-y|^{\frac{N}{2}+s}} \quad {\mbox{a.e. in~$\R^{2N}$}}.
\end{equation}
To prove this, let~$w\in C^\infty_c(\R^{2N})$ 
with 
\begin{equation}\label{deioctu45boi76iu5t4}
{\mbox{supp$(w)\cap \{(x,y)\in \R^{2N}\mbox{ s.t. } x= y \}=\varnothing$.}}\end{equation} 
For short, we let~$S:=$supp$(w)$.
By the strong convergence in~\eqref{covergenzadebole00}, we have that~$u_n(x)-u_n(y)$
converges in~$L^1(S)$, and therefore (see e.g. Theorem~4.9 in~\cite{MR2759829}) we have that~$|u_n(x)-u_n(y)|\le g(x,y)$, with~$g\in L^1(S)$.
This, \eqref{sevvbhg5847y750945uhbuf},
\eqref{deioctu45boi76iu5t4} and the Dominated Convergence Theorem
entail that
$$ \lim_{n\to +\infty}\iint_{\R^{2N}}\xi_n(x,y)\,w(x,y)\,dx\,dy=\iint_{\R^{2N}}\frac{u(x)-u(y)}{|x-y|^{\frac{N}{2}+s}}\,w(x,y)\,dx\,dy.
$$
{F}rom this and the weak convergence in~\eqref{diory867684576hgdsjg},
we obtain that
$$\iint_{\R^{2N}}\xi(x,y)\,w(x,y)\,dx\,dy=\iint_{\R^{2N}}\frac{u(x)-u(y)}{|x-y|^{\frac{N}{2}+s}}\,w(x,y)\,dx\,dy,$$
which gives the desired result in~\eqref{asdfghjkqwertyuio23456789}.

Thanks to~\eqref{AL:1} and~\eqref{asdfghjkqwertyuio23456789}
we conclude that
$$ \lim_{n\to +\infty}\iint_{\R^{2N}}
\frac{(u_n(x)-u_n(y))(v(x)-v(y))}{|x-y|^{N+2s}}\,dx\,dy 
=\iint_{\R^{2N}}
\frac{(u(x)-u(y))(v(x)-v(y))}{|x-y|^{N+2s}}\,dx\,dy .$$
Combining this information with~\eqref{servetutyepn090909}, \eqref{AL:2} and~\eqref{AL:3}, we gather that~$\langle f'_\varepsilon(u), v \rangle=0$.

Now, to complete the proof of Lemma~\ref{wea}, we perform a density argument.
Indeed, if~$v\in \mathcal{D}^{s,2}(\R^N)$, there exists a 
sequence~$\{v_n\}_n$ in~$C^\infty_c(\R^N)$ such that~$v_n$ converges to~$v$ strongly in~$\mathcal{D}^{s,2}(\R^N)$.
We have that
\begin{equation}\label{vbsbguro123456789}\langle f'_\varepsilon(u), v \rangle=
\lim_{n\to+\infty}\langle f'_\varepsilon(u), v \rangle
=\lim_{n\to+\infty}\langle f'_\varepsilon(u), v-v_n \rangle
+\langle f'_\varepsilon(u), v_n \rangle=
\lim_{n\to+\infty}\langle f'_\varepsilon(u), v-v_n \rangle.
\end{equation}
Also, we observe that, by the fractional Sobolev embedding and the H\"older inequality, 
\[
\begin{aligned}
&\langle f'_\varepsilon(u), v_n-v \rangle\\
&
\leq[u]_s[v_n-v]_{s}
-\varepsilon\int_{\R^N} h(x)u_+^q(x)(v_n(x)-v(x))\,dx
-\int_{\R^N}u_+^{2^*_s-1}(x)(v_n(x)-v(x))\,dx \\
&\leq [u]_s[v_n-v]_{s}+c_1\|v_n-v\|_{2^*_s}
+c_2\|v_n-v\|_{2^*_s} \\
&\leq [u]_s[v_n-v]_{s}+c_3[v_n-v]_{s},
\end{aligned}
\] 
for some~$c_1$, $c_2$, $c_3>0$. 

This entails that 
\[
\lim_{n\to +\infty}\langle f'_\varepsilon(u), v_n-v \rangle=0,
\]
which, together with~\eqref{vbsbguro123456789},
gives that~$\langle f'_\varepsilon(u), v \rangle=0$.

The sign of~$u$ has been commented on at the end of Section~\ref{NOTA}.
This completes the proof of Lemma~\ref{wea}, as desired.
\end{proof}

Our objective is now to prove that the functional~$f_\varepsilon$ 
defined in~\eqref{deffepsilon} satisfies the Palais-Smale 
condition below a suitable level. This goal will be accomplished in the forthcoming Proposition~\ref{PSfepsilon}.
To this end, we need to identify the appropriate energy level,
and this operation will rely on the following observation:

\begin{lem}\label{Ag78}
Let
$$r:=\frac{2^*_s}{2^*_s-q-1}.$$ Then, there exists~$C_\star>0$, depending only on~$N$, $s$, $q$
and~$\|h\|_{L^r(\R^N)}$, such that, for any~$a>0$,
\[
\frac{s}{N}a^{2^*_s}-\varepsilon\left(\frac{1}{q+1}-\frac{1}{2} \right)\|h\|_{L^r(\R^N)} \,a^{q+1}\geq -C_\star\,\varepsilon^r.
\]
\end{lem}

\begin{proof} Since~$2^*_s>q+1$, given~$\alpha$, $\beta>0$, we have that
$$ c(\alpha,\beta,s,q):=\sup_{x>0}\big( \beta x^{q+1}-\alpha x^{2^*_s}\big)<+\infty.$$
In particular, choosing~$x:=\varepsilon^{-\frac{r}{2^*_s}}a$,
\begin{eqnarray*}
\varepsilon^r c(\alpha,\beta,s,q)\ge \beta \varepsilon^r x^{q+1}-\alpha \varepsilon^r x^{2^*_s}
=\beta \varepsilon^{r-\frac{r(q+1)}{2^*_s}}a^{q+1}-\alpha a^{2^*_s}=\beta \varepsilon a^{q+1}-\alpha a^{2^*_s},
\end{eqnarray*}
hence the desired result follows by choosing
\begin{equation*}C_\star\ge
c\left(\frac{s}N,\left(\frac{1}{q+1}-\frac{1}{2} \right)\|h\|_{L^r(\R^N)},s,q\right).\qedhere\end{equation*}
\end{proof} 

The notation in Lemma~\ref{Ag78} will be implicitly assumed from now on.

\begin{prop}\label{PSfepsilon}
Suppose that~$(h_0)$ holds true and let~$\{u_n\}_n \subset \mathcal{D}^{s,2}(\R^N)$ be such that
\begin{equation}\label{PS-i}
\displaystyle{\lim_{n\to +\infty}}f_\varepsilon(u_n)=c_\varepsilon<\frac{s}{N}S^\frac{N}{2s}
-C_\star\varepsilon^r\end{equation}
and
\begin{equation}\label{PS-ii}
{\lim_{n\to +\infty}}\sup_{{v\in\mathcal{D}^{s,2}(\R^N)}\atop{[v]_s\le1}}\big|\langle f'_\varepsilon(u_n), v\rangle\big|= 0.
\end{equation}

Then, up to a subsequence, $\{u_n\}_n$ strongly converges in~$\mathcal{D}^{s,2}(\R^N)$ as~$n\to+\infty$ to a nonnegative function~$u$ that is a solution of~\eqref{problema}.
\end{prop}

\begin{proof}
Let~$\lambda \in (2,2^*_s)$. 
By the H\"older inequality with exponents~$r$ and~$\frac{2^*_s}{q+1}$,
\begin{eqnarray*}
&&\lambda f_\varepsilon(u_n)-\langle f'_\varepsilon(u_n),u_n\rangle \\
&&\quad= \left(\frac{\lambda}{2}-1 \right)[u_n]_s^2
-\varepsilon \left(\frac{\lambda}{q+1}-1 \right)\int_{\R^N}h(x)(u_n)_+^{q+1}\,dx -\left(\frac\lambda{2^*_s}-1\right)
\int_{\R^N}(u_n)_+^{2^*_s}(x)\,dx 
\\
&&\quad\geq \left(\frac{\lambda}{2}-1 \right)[u_n]_s^2
-\varepsilon \left(\frac{\lambda}{q+1}-1 \right)\int_{\R^N}h(x)(u_n)_+^{q+1}(x)\,dx \\
&&\quad\geq \left(\frac{\lambda}{2}-1 \right)[u_n]_s^2
-\varepsilon \left(\frac{\lambda}{q+1}-1 \right)\|h\|_{L^r(\R^N)}\|u_n\|_{L^{2^*_s}(\R^N)}^{q+1} \\
&&\quad\geq \left(\frac{\lambda}{2}-1 \right)[u_n]_s^2
-C\varepsilon \left(\frac{\lambda}{q+1}-1 \right)\|h\|_{L^r(\R^N)}[u_n]_s^{q+1},
\end{eqnarray*}
for some~$C>0$ depending on~$N$ and~$s$, as given by the
fractional Sobolev embedding (see e.g. Theorem~6.5
in~\cite{MR2944369}).

This, \eqref{PS-i} and~\eqref{PS-ii} yield that~$\{u_n\}_n$ is bounded in~$\mathcal{D}^{s,2}(\R^N)$, and therefore so is~$\{(u_n)_+\}_n$, because,
for all~$x$, $y\in\R^N$,
\begin{equation}\label{d4yt4fhjkagf5674839jijhgthuikio}
|(u_n)_+(x)-(u_n)_+(y)|\le |u_n(x)-u_n(y)|.\end{equation}

Therefore, up to a subsequence,
\begin{equation}\label{covergenzadebole}\begin{split}
&u_n \rightharpoonup u \mbox{ weakly in }\mathcal{D}^{s,2}(\R^N),
\\ &u_n\rightarrow u \mbox{ in } L^r_{loc}({\R^N}) \;{\mbox{ for every~$r\in [1,2^*_s)$,}}
\\&
u_n\to u \mbox{ a.e. in } \R^N,\\
&(u_n)_+ \rightharpoonup \widetilde u \mbox{ weakly in }\mathcal{D}^{s,2}(\R^N),
\\ &(u_n)_+\rightarrow \widetilde u \mbox{ in } L^r_{loc}({\R^N}) \;{\mbox{ for every~$r\in [1,2^*_s)$}}
\\ \mbox{and}\qquad&
(u_n)_+\to \widetilde u \mbox{ a.e. in } \R^N.
\end{split}\end{equation}
We point out that, for a.e.~$x\in\R^N$,
\begin{eqnarray*}&&
\widetilde u(x)=\lim_{n\to+\infty} (u_n)_+(x)
=\lim_{n\to+\infty} \max\{u_n(x),0\}
\\&&\qquad= \max\left\{\lim_{n\to+\infty} u_n(x),0\right\}=\max\{u(x),0\}=u_+(x).
\end{eqnarray*}

We can now apply Theorem~1.1 in~\cite{MR3866572} to find that 
there exist two bounded measures~$\mu$ and~$\nu$, an at most 
countable set of indices~$I$, and positive real numbers~$\mu_i$ and~$\nu_i$, with~$i\in I$, such that the following convergence 
holds weakly* in the sense of measures:
\begin{align}&
|D^s (u_n)_+|^2\,dx \rightharpoonup \mu\geq |D^s u_+|^2\,dx
+\sum_{i\in I}\mu_i \delta_{x_i} \label{CCmui}\\
{\mbox{and }}\qquad&|(u_n)_+|^{2^*_s}\,dx \rightharpoonup \nu= |u_+|^{2^*_s}\,dx
+\sum_{i\in I}\nu_i \delta_{x_i}. \label{CCnui} 
\end{align} Moreover, for all~$ i \in I$,
\begin{equation}\label{CCSi}
S^\frac{1}{2}\nu_i^\frac{1}{2^*_s}\leq \mu_i^\frac{1}{2},
\end{equation}
where~$S$ is defined in~\eqref{definizioneS}.

Furthermore, Theorem 1.1 in~\cite{MR3866572} also gives that, defining
\begin{align}\label{defmuinfinito}
&\mu_\infty:=\lim_{R\to +\infty}\limsup_{n\to +\infty}
\int_{\R^N\setminus B_R}|D^s(u_n)_+(x)|^2\,dx
\\{\mbox{and }}\qquad&
\label{defnuinfinito}
\nu_\infty:=\lim_{R\to +\infty}\limsup_{n\to +\infty}
\int_{\R^N\setminus B_R}|(u_n)_+(x)|^{2^*_s}\,dx,
\end{align}
we have that
\begin{align}\label{CCmuinfinito}
&\limsup_{n\to +\infty}
\int_{\R^N}|D^s(u_n)_+(x)|^2\,dx=\mu(\R^N)+\mu_\infty \\
{\mbox{and }}\qquad&
\label{CCnuinfinito}
\limsup_{n\to +\infty}
\int_{\R^N}|(u_n)_+(x)|^{2^*_s}\,dx=\nu(\R^N)+\nu_\infty,\\
{\mbox{with }} \qquad &
\label{CCSinfinito}
S^\frac{1}{2}\nu_\infty^\frac{1}{2^*_s}\leq \mu_\infty^\frac{1}{2}.
\end{align}

Now we claim that
\begin{equation}\label{claimaggpergh75}
{\mbox{if~$\mu_i\not=0$ for some~$i\in I$, then~$\nu_i\ge S^\frac{N}{2s}$.}}
\end{equation}
To establish this claim, let~$\delta>0$ and, recalling~\eqref{definizionephi}, consider~$\phi_\delta:=\phi_{\delta,x_i}$. Notice that~$\phi_\delta u_n
\in \mathcal{D}^{s,2}(\R^N)$, and therefore, in light of~\eqref{PS-ii},
we have that
\begin{equation}\label{scontrofdelta}
\begin{aligned}
0&=\lim_{n\to +\infty}\langle f_\varepsilon'(u_n),\phi_\delta u_n
\rangle \\
&=\lim_{n\to +\infty}\Bigg( \int_{\R^N}\phi_\delta(x)|D^su_n(x)|^2\,dx
\\&\qquad\qquad\qquad+\iint_{\R^{2N}}u_n(y)\frac{(u_n(x)-u_n(y))(\phi_\delta(x)-\phi_\delta(y))}{|x-y|^{N+2s}}\,dx\,dy \\
&\qquad\qquad\qquad-\varepsilon \int_{\R^N} h(x)\phi_\delta(u_n)_+^{q+1}(x)\,dx
-\int_{\R^N} \phi_\delta(x)(u_n)_+^{2^*_s}(x)\,dx \Bigg).
\end{aligned}
\end{equation}

Moreover, it follows from~\eqref{CCmui} and~\eqref{CCnui} that
\begin{align}\label{limitephideltaDu}
&\lim_{n\to +\infty} \int_{\R^N}\phi_\delta(x)|D^s(u_n)_+(x)|^2 \,dx
=\int_{\R^N} \phi_\delta\,d\mu,
\\ {\mbox{and }}\qquad&
\label{limitephideltau}
\lim_{n\to +\infty} \int_{\R^N} \phi_\delta(x)|(u_n)_+(x)|^{2^*_s}\,dx
=\int_{\R^N} \phi_\delta\,d\nu.
\end{align}
Also, since supp$(\phi_\delta)$ is bounded, 
the convergence in~\eqref{covergenzadebole} gives that,
up to a subsequence,
\begin{equation}\label{limitehphideltau}
\lim_{n\to +\infty}\int_{\R^N} h(x)\phi_\delta(x)(u_n)_+^{q+1}(x)\,dx
=\int_{B_\delta(x_i)} h(x)\phi_\delta(x) u_+^{q+1}(x)\,dx.
\end{equation}

In addition, by the H\"older inequality,
\begin{eqnarray*}\left|
\iint_{\R^{2N}}u_n(y)\frac{(u_n(x)-u_n(y))(\phi_\delta(x)-\phi_\delta(y))}{|x-y|^{N+2s}}\,dx\,dy \right|&\leq& [u_n]_s
\left(\int_{\R^N}|u_n|^2|D^s\phi_\delta|^2\,dx \right)^\frac{1}{2}\\
&\leq& M \left(\int_{\R^N}|u_n|^2|D^s\phi_\delta|^2\,dx \right)^\frac{1}{2},
\end{eqnarray*}
for some~$M>0$.

Hence, using Lemma~\ref{lemma2.4} (with~$w:=|D^s\phi_\delta|^2$,
which satisfies the assumption in~\eqref{fhewuoguo46859432},
thanks to Lemma~\ref{corollary2.3}), we conclude that
\begin{equation}\label{deltato0}\begin{split}&
\limsup_{n\to +\infty}\left|\iint_{\R^{2N}}u_n(y)\frac{(u_n(x)-u_n(y))(\phi_\delta(x)-\phi_\delta(y))}{|x-y|^{N+2s}}\,dx\,dy\right|
\\&\qquad\qquad\le M \left(\int_{\R^N}|u(x)|^2|D^s\phi_\delta(x)|^2\,dx \right)^\frac{1}{2}.\end{split}\end{equation}
Besides, from Lemma~\ref{lemmalimiteDphidelta},
$$ \lim_{\delta\to 0} \int_{\R^N}|u(x)|^2|D^s\phi_\delta(x)|^2\,dx=0,$$
which, combined with~\eqref{deltato0}, returns that
$$ \lim_{\delta\to 0}\lim_{n\to +\infty}\iint_{\R^{2N}}u_n(y)\frac{(u_n(x)-u_n(y))(\phi_\delta(x)-\phi_\delta(y))}{|x-y|^{N+2s}}\,dx\,dy=0.$$

{F}rom this, \eqref{scontrofdelta}, \eqref{limitephideltaDu},
\eqref{limitephideltau} and~\eqref{limitehphideltau}
(and recalling also~\eqref{d4yt4fhjkagf5674839jijhgthuikio}),
it follows that
\begin{eqnarray*}
0&=&\lim_{\delta\to 0}\lim_{n\to  +\infty}
\langle f_\varepsilon'(u_n),\phi_\delta u_n\rangle
\\ &\ge&\lim_{\delta\to 0}\left( \int_{\R^N}\phi_\delta\,d\mu
-\varepsilon \int_{B_\delta(x_i)} h(x)\phi_\delta(x) u_+^{q+1}(x)\,dx
-\int_{\R^N} \phi_\delta\,d\nu \right)\\
&\ge&\lim_{\delta\to 0}\left( \int_{\R^N}\phi_\delta(x)|D^su_+(x)|^2\,dx
+\sum_{j\in I}\mu_j \phi_\delta(x_j)
-\int_{\R^N} \phi_\delta(x)u_+^{2^*_s}(x)\,dx-\sum_{j\in I}\nu_j \phi_\delta(x_j) \right)
\\&=&\mu_i-\nu_i.
\end{eqnarray*}
This gives that~$\nu_i\ge\mu_i$.

Since~$\mu_i\neq0$ (recall the assumption in~\eqref{claimaggpergh75}),
this implies that~$\nu_i\neq0$ and that, by~\eqref{CCSi},
$$ S^{\frac12}\nu_i^{\frac1{2^*_s}}\le \mu_i^{\frac12}\le\nu_i^{\frac12},$$ leading to~$\nu_i\ge S^\frac{N}{2s}$,
which completes the proof of~\eqref{claimaggpergh75}.

Similarly, we claim that
\begin{equation}\label{claimaggpergh75BIS}
{\mbox{if $\mu_\infty\neq0$, then~$\nu_\infty \geq S^\frac{N}{2s}$.}}
\end{equation}
To this end, we let~$R>0$ and,
recalling~\eqref{definizionephi}, set~$\phi_R:=1-\phi_{R,0}$.
Notice that~$\phi_R u_n\in \mathcal{D}^{s,2}(\R^N)$, and therefore, in light of~\eqref{PS-ii},
\begin{equation}\label{scontrofR}
\begin{aligned}
0&=\lim_{n\to +\infty}\langle f_\varepsilon'(u_n),\phi_R u_n
\rangle \\
&=\lim_{n\to +\infty}\Bigg( \int_{\R^N}\phi_R(x)|D^su_n(x)|^2\,dx
\\&\qquad\qquad\qquad+\iint_{\R^{2N}}u_n(y)\frac{(u_n(x)-u_n(y))(\phi_R(x)-\phi_R(y))}{|x-y|^{N+2s}}\,dx\,dy \\
&\qquad\qquad\qquad-\varepsilon \int_{\R^N} h(x)\phi_R(x)(u_n)_+^{q+1}(x)\,dx
-\int_{\R^N} \phi_R(x)(u_n)_+^{2^*_s}(x)\,dx \Bigg).
\end{aligned}
\end{equation}

Recalling that~$\{u_n\}_n$ is bounded in~$\mathcal{D}^{s,2}(\R^N)$ and using the fractional Sobolev embedding and the H\"older inequality,
\begin{eqnarray*}
&&\left|\int_{\R^N} h(x)\phi_R(x)(u_n)_+^{q+1}(x)\,dx\right|
\le \int_{\R^N\setminus B_{R/2}}|h(x)|\, |u_n(x)|^{q+1}\,dx
\\&&\qquad\leq \|u_n\|_{L^{2^*_s}(\R^N)}^{q+1}
\left(\int_{\R^N\setminus B_R}|h(x)|^{\frac{2^*_s}{2^*_s-q-1}}\,dx
\right)^{\frac{2^*_s-q-1}{2^*_s}} \\
&&\qquad \leq C[u_n]_s^{q+1} \left(\int_{\R^N\setminus B_{R/2}}|h(x)|^{\frac{2^*_s}{2^*_s-q-1}}\,dx
\right)^{\frac{2^*_s-q-1}{2^*_s}}\\
&&\qquad \leq C\left(\int_{\R^N\setminus B_{R/2}}|h(x)|^{\frac{2^*_s}{2^*_s-q-1}}\,dx
\right)^{\frac{2^*_s-q-1}{2^*_s}},
\end{eqnarray*}
up to renaming~$C>0$ from line to line, and therefore
\begin{equation}\label{limitehphiRu}
\lim_{R\to +\infty} \lim_{n\to +\infty}\int_{\R^N}h(x)\phi_R(x) (u_n)_+^{q+1}(x)\,dx=0.
\end{equation}
 
Furthermore, by the H\"older inequality,
\begin{eqnarray*}&&\left|
\iint_{\R^{2N}}u_n(y)\frac{(u_n(x)-u_n(y))(\phi_R(x)-\phi_R(y))}{|x-y|^{N+2s}}\,dx\,dy\right|\\&&\qquad\qquad
\leq [u_n]_s
\left(\int_{\R^N}|u_n(x)|^2|D^s\phi_R(x)|^2\,dx \right)^\frac{1}{2}\\
&&\qquad\qquad\leq M \left(\int_{\R^N}|u_n(x)|^2|D^s\phi_R(x)|^2\,dx \right)^\frac{1}{2} 
\end{eqnarray*}
for some~$M>0$.
Hence, by virtue of Lemma~\ref{lemma2.4},
\begin{equation*}\begin{split}& \limsup_{n\to+\infty}\left|\iint_{\R^{2N}}u_n(y)\frac{(u_n(x)-u_n(y))(\phi_R(x)-\phi_R(y))}{|x-y|^{N+2s}}\,dx\,dy\right|\\&\qquad\qquad
\le M \left(\int_{\R^N}|u(x)|^2|D^s\phi_R(x)|^2\,dx \right)^\frac{1}{2}.
\end{split}\end{equation*}
Since from Lemma~\ref{lemmalimiteDphiR} we know that
\begin{equation*}
\lim_{R\to +\infty}\int_{\R^N}|u(x)|^2|D^s\phi_R(x)|^2\,dx=0,
\end{equation*}
we conclude that
\begin{equation}\label{mnbvcxzsdfghjkoiuytr}
\lim_{R\to +\infty} \lim_{n\to+\infty}\iint_{\R^{2N}}u_n(y)\frac{(u_n(x)-u_n(y))(\phi_R(x)-\phi_R(y))}{|x-y|^{N+2s}}\,dx\,dy=0.\end{equation}

Additionally, given two sequences~$\{a_n\}_n$ and~$\{b_n\}_n$, we know\footnote{To check~\eqref{funcdiseg33672547}, we notice that
\begin{eqnarray*}
&&\limsup_{n\to+\infty}a_n=\limsup_{n\to+\infty}\Big( (a_n-b_n)+b_n\Big)
\le \limsup_{n\to+\infty}(a_n-b_n) + \limsup_{n\to+\infty}b_n.
\end{eqnarray*}
} that
\begin{equation}\label{funcdiseg33672547}
\limsup_{n\to+\infty}\big(a_n-b_n\big)\ge \limsup_{n\to+\infty}a_n-
\limsup_{n\to+\infty}b_n.\end{equation}

Therefore, in view of this inequality,
used here with
$$ a_n:=\int_{\R^N}\phi_R(x)|D^s(u_n)_+(x)|^2\,dx\qquad
{\mbox{and}}\qquad \int_{\R^N}\phi_R(x) (u_n)_+^{2^*_s}(x)\,dx,$$
and recalling~\eqref{defmuinfinito}, \eqref{defnuinfinito}, \eqref{scontrofR}, \eqref{limitehphiRu} and~\eqref{mnbvcxzsdfghjkoiuytr},
we see that
\[
\begin{aligned}
0&=\lim_{R\to +\infty} \limsup_{n\to +\infty}\left(\int_{\R^N}\phi_R(x)|D^su_n(x)|^2\,dx- \int_{\R^N}\phi_R(x) (u_n)_+^{2^*_s}(x)\,dx\right) \\
&\ge\lim_{R\to +\infty} \limsup_{n\to +\infty}\left( \int_{\R^N}\phi_R(x)|D^s(u_n)_+(x)|^2\,dx
-\int_{\R^N} \phi_R(x)(u_n)_+^{2^*_s}(x)\,dx \right)\\
&\ge\lim_{R\to +\infty} \left( \limsup_{n\to +\infty}\int_{\R^N}\phi_R(x)|D^s(u_n)_+(x)|^2\,dx
-\limsup_{n\to +\infty}\int_{\R^N} \phi_R(x)(u_n)_+^{2^*_s}(x)\,dx \right)\\
&=\mu_\infty- \nu_\infty,
\end{aligned}
\]
which implies that~$\mu_\infty\le\nu_\infty$.

Hence, if~$\mu_\infty\neq0$, then~$\nu_\infty\neq0$ as well.
Accordingly, by~\eqref{CCSinfinito},
$$
S^\frac{1}{2}\nu_\infty^\frac{1}{2^*_s}\leq \mu_\infty^\frac{1}{2}\le\nu_\infty^{\frac12},$$
and thus~$\nu_\infty\ge S^{\frac{N}{2s}}$,
which establishes the desired claim in~\eqref{claimaggpergh75BIS}.

Our next observation is that
\begin{equation}\label{BIS480BIS}
{\mbox{$\nu_i=0$, for all~$i\in I$, and~$\nu_\infty=0$.}}
\end{equation}
To prove this, we use the assumptions in~\eqref{PS-i} and~\eqref{PS-ii} to see that
\begin{equation}\label{aggpercheserve342748}\begin{split}
c_\varepsilon
&=\lim_{n\to +\infty}\left(f_\varepsilon(u_n)
-\frac{1}{2}\langle f_\varepsilon'(u_n), u_n \rangle\right) \\
&=\lim_{n\to +\infty}\Bigg[\left(\frac{1}{2}\int_{\R^{N}}|D^s u_n(x)|^2
\,dx
-\frac{\varepsilon}{q+1}\int_{\R^N}h(x)(u_n)_+^{q+1}(x)\,dx
-\frac{1}{2^*_s}\int_{\R^N}(u_n)_+^{2^*_s}(x)\,dx
\right)\\
&\qquad\qquad-\frac12\left(\int_{\R^{N}}|D^s u_n(x)|^2\,dx
-\varepsilon \int_{\R^N}h(x)(u_n)_+^{q+1}(x)\,dx
-\int_{\R^N}(u_n)_+^{2^*_s}(x)\,dx\right)\Bigg]\\
&=\lim_{n\to +\infty}\Bigg[\varepsilon\left(\frac12
-\frac{1}{q+1}\right)\int_{\R^N}h(x)(u_n)_+^{q+1}(x)\,dx
+\left(\frac12
-\frac{1}{2^*_s}\right)\int_{\R^N}(u_n)_+^{2^*_s}(x)\,dx\Bigg]
.
\end{split}\end{equation}

Moreover, we claim that
\begin{equation}\label{limconvergenzadebole}
\lim_{n\to +\infty}\int_{\R^N} h(x)(u_n)_+^{q+1}(x)\,dx
=\int_{\R^N} h(x)u_+^{q+1}(x)\,dx.
\end{equation}
To this end, let~$R>0$ and notice that
\begin{eqnarray*}
\lim_{n\to +\infty}\int_{B_R} h(x)(u_n)_+^{q+1}(x)\,dx
=\int_{B_R} h(x)u_+^{q+1}(x)\,dx,
\end{eqnarray*}
thanks to the strong convergence in~\eqref{covergenzadebole}.

Also, by the H\"older inequality with exponents~$\frac{2^*_s}{2^*_s-q-1}$ and~$\frac{2^*_s}{q+1}$,
\begin{eqnarray*}\left|
\int_{\R^N\setminus B_R} h(x)(u_n)_+^{q+1}(x)\,dx\right|&\le&
\left( \int_{\R^N\setminus B_R} |h(x)|^{\frac{2^*_s}{2^*_s-q-1}}\,dx\right)^{\frac{2^*_s-q-1}{2^*_s}}
\left(\int_{\R^N\setminus B_R} (u_n)_+^{2^*_s}(x)\,dx\right)^{\frac{q+1}{2^*_s}}\\&\le& C \left( \int_{\R^N\setminus B_R} |h(x)|^{\frac{2^*_s}{2^*_s-q-1}}\,dx\right)^{\frac{2^*_s-q-1}{2^*_s}},
\end{eqnarray*}
yielding that
$$ \lim_{R\to +\infty}
\lim_{n\to +\infty}\int_{\R^N\setminus B_R} h(x)(u_n)_+^{q+1}(x)\,dx=0.$$
As a result,
\begin{eqnarray*}&&
\lim_{n\to +\infty}\int_{\R^N} h(x)(u_n)_+^{q+1}(x)\,dx=
\lim_{R\to +\infty}\lim_{n\to +\infty}\int_{\R^N} h(x)(u_n)_+^{q+1}(x)\,dx\\
&&\qquad=
\lim_{R\to +\infty}\lim_{n\to +\infty}\left(\int_{B_R} h(x)(u_n)_+^{q+1}(x)\,dx
+\int_{\R^N\setminus B_R} h(x)(u_n)_+^{q+1}(x)\,dx\right)
\\
&&\qquad= \lim_{R\to+\infty} \int_{B_R} h(x)u_+^{q+1}(x)\,dx\\
&&\qquad = \int_{\R^N} h(x)u_+^{q+1}(x)\,dx,
\end{eqnarray*}
which is~\eqref{limconvergenzadebole}.

Using~\eqref{limconvergenzadebole} into~\eqref{aggpercheserve342748},
we obtain that
\begin{equation}\label{aggpercheserve342748BIS}\begin{split}
c_\varepsilon
&=\varepsilon\left(\frac12
-\frac{1}{q+1}\right)\int_{\R^N}h(x)u_+^{q+1}(x)\,dx
+\left(\frac12
-\frac{1}{2^*_s}\right)\lim_{n\to +\infty}\int_{\R^N}(u_n)_+^{2^*_s}(x)\,dx.\end{split}\end{equation}
In addition, by~\eqref{CCnui} and~\eqref{CCnuinfinito}, we see that
\begin{equation*}\begin{split}
&\limsup_{n\to +\infty}\int_{\R^N} (u_n)_+^{2^*_s}(x)\,dx
=\nu\left(\R^N \right)+\nu_\infty =
\int_{\R^N} u_+(x)^{2^*_s}\,dx+\sum_{i\in I}\nu_i+\nu_\infty.
\end{split}\end{equation*}
By plugging this into~\eqref{aggpercheserve342748BIS},
we conclude that
\begin{equation}\label{chepallefefew4rt74}\begin{split}
c_\varepsilon
&=\varepsilon\left(\frac12
-\frac{1}{q+1}\right)\int_{\R^N}h(x)u_+^{q+1}(x)\,dx
+\left(\frac12-\frac{1}{2^*_s}\right)
\left(\int_{\R^N} u_+(x)^{2^*_s}\,dx+\sum_{i\in I}\nu_i+\nu_\infty\right)
.\end{split}\end{equation}

Moreover, by the H\"older inequality with exponents~$\frac{2^*_s}{q+1}$
and~$\frac{2^*_s}{2^*_s-q-1}$, we get that
\begin{eqnarray*} \int_{\R^N}h(x)u_+^{q+1}(x)\,dx&\le&
\left(\int_{\R^N}|h(x)|^{\frac{2^*_s}{2^*_s-q-1}}\,dx\right)^{\frac{2^*_s-q-1}{2^*_s}}\left(\int_{\R^N}u_+^{2^*_s}(x)\,dx\right)^{\frac{q+1}{2^*_s}}
\\&=&\|h\|_{L^{\frac{2^*_s}{2^*_s-q-1}}(\R^N)}\|u\|_{L^{2^*_s}(\R^N)}^{q+1}.\end{eqnarray*}
This and~\eqref{chepallefefew4rt74} entail that
\begin{equation*} \begin{split}
c_\varepsilon
&\ge \varepsilon\left(\frac12
-\frac{1}{q+1}\right)\|h\|_{L^{\frac{2^*_s}{2^*_s-q-1}}(\R^N)}\|u\|_{L^{2^*_s}(\R^N)}^{q+1}
+\left(\frac12-\frac{1}{2^*_s}\right)
\left(\int_{\R^N} u_+(x)^{2^*_s}\,dx+\sum_{i\in I}\nu_i+\nu_\infty\right)\\
&=\varepsilon\left(\frac12
-\frac{1}{q+1}\right)\|h\|_{L^{\frac{2^*_s}{2^*_s-q-1}}(\R^N)}\|u\|_{L^{2^*_s}(\R^N)}^{q+1}
+\frac{s}{N}\|u\|_{L^{2^*_s}(\R^N)}^{2^*_s}
+\frac{s}{N}
\left(\sum_{i\in I}\nu_i+\nu_\infty\right)
.\end{split}\end{equation*}
In this way, we can employ Lemma~\ref{Ag78},
with~$a:=\|u\|_{L^{2^*_s}(\R^N)}$, obtaining that
\begin{equation}\label{fwgb5476i76u5}
c_\varepsilon\ge\frac{s}{N}
\left(\sum_{i\in I}\nu_i+\nu_\infty\right)
-C_\star\varepsilon^r.
\end{equation}

Let now~$\kappa_1\in\N\cup\{+\infty\}$ be the number of indices~$i\in I$ such that~$\nu_i\neq0$ and
$$ \kappa_2:=\begin{cases}
1 \quad{\mbox{ if }} \nu_\infty\neq0,\\
0 \quad{\mbox{ if }} \nu_\infty=0.
\end{cases}$$
We observe that, thanks to~\eqref{CCSi} and~\eqref{claimaggpergh75},
\begin{eqnarray*}
\sum_{i\in I}\nu_i=\sum_{{i\in I}\atop{\nu_i\neq0}}\nu_i
\ge \sum_{{i\in I}\atop{\nu_i\neq0}}S^{\frac{N}{2s}}=\kappa_1 S^{\frac{N}{2s}}.
\end{eqnarray*}
Also, using~\eqref{CCSinfinito} and~\eqref{claimaggpergh75BIS},
$$\nu_\infty =\kappa_2\nu_\infty\ge \kappa_2
S^{\frac{N}{2s}}.$$
{F}rom these observations and~\eqref{fwgb5476i76u5}, we conclude that
$$ c_\varepsilon\ge\frac{s}{N}(\kappa_1+\kappa_2)S^{\frac{N}{2s}}
-C_\star\varepsilon^r.$$
Combining this with~\eqref{PS-i}, we conclude that
$$\frac{S^\frac{N}{2s}\,s}N-C_\star\varepsilon^r
>\frac{s}{N}(\kappa_1+\kappa_2)S^{\frac{N}{2s}}-C_\star\varepsilon^r$$
and therefore
$$ \kappa_1+\kappa_2<1.
$$
This gives that~$\kappa_1=\kappa_2=0$, which
completes the proof of the claim in~\eqref{BIS480BIS}.

In light of~\eqref{CCnui} and~\eqref{BIS480BIS}, we conclude that
\begin{equation}\label{forsesperiamo}
\lim_{n\to +\infty}\int_{\R^N}(u_n)_+^{2^*_s}(x)\,dx
=\int_{\R^N}u_+^{2^*_s}(x)\,dx.
\end{equation}

Now, thanks to~\eqref{PS-ii} and Lemma~\ref{wea}, recalling also
the convergence in~\eqref{limconvergenzadebole} and~\eqref{forsesperiamo}, we have that
\[
\begin{aligned}
0&=\lim_{n\to +\infty}\langle f_\varepsilon'(u_n), u_n \rangle
-\langle f_\varepsilon'(u), u \rangle \\
&=\lim_{n\to +\infty}\left([u_n]_s
-\varepsilon\int_{\R^N}h(x)(u_n)_+^{q+1}(x)
-\int_{\R^N}(u_n)_+^{2^*_s}(x)\,dx\right) \\
&\qquad -[u]_s
+\varepsilon\int_{\R^N}h(x)u_+^{q+1}(x)
+\int_{\R^N}u_+^{2^*_s}(x)\,dx \\
&=\lim_{n\to +\infty}[u_n]_s-[u]_s
\end{aligned},
\]
that is
\[
\lim_{n\to +\infty}[u_n]_s=[u]_s.
\]
This and the uniform convexity of~$\mathcal{D}^{s,2}(\R^N)$
give that~$u_n$ converges strongly to~$u$ in~$\mathcal{D}^{s,2}(\R^N)$, as~$n\to+\infty$, as desired.
\end{proof}

We can now prove the existence of a nonnegative solution~$u_\varepsilon$ of~\eqref{problema} such that~$u_\varepsilon \to 0$ as~$\varepsilon \to 0$.

\begin{prop}\label{prop:primasol}
Let $(h_0)$ and $(h_1)$ hold true.

Then, there exists~$\varepsilon_0>0$ such that, if~$\varepsilon\in(0,\varepsilon_0)$,
the functional~$f_\varepsilon$ achieves a nonnegative local minimum~$u_\varepsilon\not\equiv 0$. 

Moreover, $u_\varepsilon\in L^\infty(\R^N)\cap C^\alpha(\R^N)$, for any~$\alpha\in(0,\min\{2s,1\})$, and~$u_\varepsilon\to 0$ 
in~$\mathcal{D}^{s,2}(\R^N)$ as~$\varepsilon \to 0$.

Also, if~$h\ge0$ in some open set~$\Omega\subseteq\R^N$, then~$u_\varepsilon(x)>0$ for all~$x\in\Omega$.

In particular,
if~$h\ge0$, then~$u_\varepsilon(x)>0$ for all~$x\in\R^N$.
\end{prop}

\begin{proof}
{F}rom the H\"older inequality and the fractional Sobolev embedding, we have that
\begin{eqnarray*}&&
\left|\int_{\R^N}h(x)u_+^{q+1}(x)\,dx\right|
\leq C\|u_+\|_{L^{2^*_s}(\R^N)}^{q+1}
\leq c_1 [u]_{s}^{q+1}
\\{\mbox{and }}&&
\int_{\R^N}u_+^{2^*_s}(x)\,dx
\leq c_2 [u]_s^{2^*_s},
\end{eqnarray*}
for some~$C$, $c_1$, $c_2>0$, and therefore, up to relabeling the constants~$c_1$ and~$c_2$,
\begin{equation}\label{Zo}
f_\varepsilon(u)\geq 
\frac{1}{2}[u]_s^2-\varepsilon c_1[u]_s^{q+1}
-c_2 [u]_s^{2^*_s}.
\end{equation}

Let now~$\rho=\rho(\varepsilon)$ be the first zero of the function
\[ (0,+\infty)\ni t\longmapsto
\frac{1}{2}t^2-\varepsilon c_1t^{q+1}
-c_2 t^{2^*_s}.
\]
We notice that~$\rho(\varepsilon) \to 0$ as~$\varepsilon \to 0$.

Furthermore, by~\eqref{Zo}, there exists~$c_0>0$ such that 
\begin{equation}\label{c_0e_0}\begin{split}
&{\mbox{if $ [u]_s<\rho$, then $
f_\varepsilon(u)\geq -c_0$}}\\
&{\mbox{and, if $ [u]_s=\rho$, then $
f_\varepsilon(u)\geq 0 $.}}\end{split}
\end{equation}

Now, let~$\phi\in C^\infty_c(\R^N)$ be a nonnegative function 
such that~$[\phi]_{s}=1$ and
supp$(\phi)\subset \{h>0 \}$. For any~$t>0$,
we have that 
\begin{eqnarray*}
f_\varepsilon(t\phi)&=&\frac{1}{2}t^2
-\frac{\varepsilon}{q+1}t^{q+1}\int_{\R^N}h(x)\phi^{q+1}(x)\,dx
-\frac{1}{2^*_s}t^{2^*_s}\int_{\R^N}\phi^{2^*_s}(x)\,dx\\
&\le&\frac{1}{2}t^2
-\frac{\varepsilon}{q+1}t^{q+1}\inf_B h\int_{B} \phi^{q+1}(x)\,dx
-\frac{1}{2^*_s}t^{2^*_s}\int_{\R^N}\phi^{2^*_s}(x)\,dx
.
\end{eqnarray*}
Hence, exploiting the assumption in~$(h_1)$,
since~$q+1<2<2^*_s$, there exists~$t_0<\rho$ 
such that, for any~$t\in(0,t_0)$, we have that~$
f_\varepsilon(t\phi)<0$.
On this account, from~\eqref{c_0e_0} it follows that
\[
-\infty< i_\varepsilon:=\inf_{{\mathcal{D}^{s,2}(\R^N)}\atop{[u]_s< \rho}}
f_\varepsilon(u)<0.
\]
Accordingly, we can take~$\varepsilon>0$ sufficiently small such that~$\frac{s}{N}S^\frac{N}{2s}-C_\star\varepsilon^r>0$, and thus
$$i_\varepsilon<
\frac{s}{N}S^\frac{N}{2s}-C_\star\varepsilon^r.$$

In this way, we see that if~$u_n$ is a minimizing sequence for~$i_\varepsilon$
then, for~$\varepsilon$ sufficiently small,
\begin{eqnarray*}
&&\lim_{n\to+\infty} f_\varepsilon(u_n)=i_\varepsilon<
\frac{s}{N}S^\frac{N}{2s}-C_\star\varepsilon^r\\{\mbox{and }}&&
{\lim_{n\to +\infty}}\sup_{{v\in\mathcal{D}^{s,2}(\R^N)}\atop{[v]_s\le1}}\langle f'_\varepsilon(u_n), v\rangle= 0,
\end{eqnarray*}
thanks also to the Ekeland's variational principle in Corollary~2.3 of~\cite{MR346619}.

This says that we are in the position of applying
Proposition~\ref{PSfepsilon} with~$c_\varepsilon:=i_\varepsilon$, deducing that, up to a subsequence, $\{u_n\}_n$ strongly converges in~$\mathcal{D}^{s,2}(\R^N)$ as~$n\to+\infty$ to a nonnegative function~$u_\varepsilon$ such that~$f_\varepsilon(u_\varepsilon)=i_\varepsilon$.

By construction, $u_\varepsilon$ is a local minimizer 
for~$f_\varepsilon$ and does not vanish identically. 

Furthermore,
\[
\lim_{\varepsilon \to 0}[u_\varepsilon]_{s}\leq 
\lim_{\varepsilon \to 0}\rho(\varepsilon)=0,
\]as desired.

We also remark that, from Proposition~5.1.1 in~\cite{MR3617721},
solutions of~\eqref{problema} are in~$ L^\infty(\R^N)$.
Also, Corollary~5.1.3 in~\cite{MR3617721} gives that~$u_\varepsilon\in C^\alpha(\R^N)$, for any~$\alpha\in(0,\min\{2s,1\})$.

Finally, we observe that if~$h\geq 0$ in~$\Omega$, then~$(-\Delta)^s u_\varepsilon\ge0$ in~$\Omega$ in the weak sense, and therefore 
the strong maximum principle for continuous weak solutions in
Proposition~5.2.1 in~\cite{MR3617721} yields that the solution~$u_\varepsilon$ is strictly positive in~$\Omega$.
\end{proof}

With this, we have constructed the first solution claimed in Theorem~\ref{thduesoluzioni}.
In the next section, we will establish the existence of a second solution.

\section{Existence of a second solution}\label{PRO2}

In this section we prove that~\eqref{problema} admits a second
solution~$\widetilde{u}_\varepsilon$. More precisely, we show that
if~$f_\varepsilon$ has a local minimum, 
then there exists a second nonnegative solution of~\eqref{problema}.

For this, from now on, we will suppose that~$\varepsilon\in(0,\varepsilon_0)$,
where~$\varepsilon_0$ is given by Proposition~\ref{prop:primasol}.
Without loss of generality, and without further notice, we will
possibly take~$\varepsilon$ even smaller whenever convenient, up to renaming~$\varepsilon_0$.

Given the nonnegative local minimum~$u_\varepsilon$ of~$f_\varepsilon$ given by Proposition~\ref{prop:primasol}, we set
\begin{equation}\label{degtfte093654}\begin{split}
g(x,t)&:
=\varepsilon h(x)\Big( (u_\varepsilon+t_+)^q-u_\varepsilon^q\Big)
+\Big((u_\varepsilon+t_+)^{2^*_s-1}-u_\varepsilon^{2^*_s-1}\Big) \\&=
\begin{cases}
\varepsilon h(x)\Big( (u_\varepsilon+t)^q-u_\varepsilon^q\Big)
+\Big((u_\varepsilon+t)^{2^*_s-1}-u_\varepsilon^{2^*_s-1}\Big) & t\geq 0,
\\
0      &t<0,
\end{cases}
\\
{\mbox{and }}\quad
G(x,v)&:=\int_0^v g(x,t)\,dt\\&=
\frac{\varepsilon}{q+1}h(x)
\Big((u_\varepsilon(x)+v_+(x))^{q+1}-u_\varepsilon^{q+1}(x)
-(q+1)u_\varepsilon^q(x) v_+(x)\Big) \\
&\qquad+\frac{1}{2^*_s}\Big((u_\varepsilon(x)+v_+(x))^{2^*_s}-u_\varepsilon^{2^*_s}(x)
-2^*_s u_\varepsilon^{2^*_s-1}(x) v_+(x)\Big)
\end{split}\end{equation}
and we define the translated functional
$I_\varepsilon:\mathcal{D}^{s,2}(\R^N) \to \R$ as
\begin{equation}\label{definizioneIepsilon}\begin{split}
I_\varepsilon(v)&:=\frac{1}{2}\int_{\R^N}|D^s v(x)|^2\,dx
-\int_{\R^N}G(x,v(x))\,dx\\&=\frac{1}{2}\int_{\R^N}|D^s v(x)|^2\,dx\\&\qquad-
\frac{\varepsilon}{q+1}\int_{\R^N}h(x)
\Big((u_\varepsilon(x)+v_+(x))^{q+1}-u_\varepsilon^{q+1}(x)
-(q+1)u_\varepsilon^q(x) v_+(x)\Big)\,dx
\\&\qquad-\frac{1}{2^*_s}\int_{\R^N}\Big((u_\varepsilon(x)+v_+(x))^{2^*_s}-u_\varepsilon^{2^*_s}(x)
-2^*_s u_\varepsilon^{2^*_s-1}(x) v_+(x)\Big)\,dx
.\end{split}
\end{equation}

We point out that, for every~$w\in\mathcal{D}^{s,2}(\R^N)$,
\begin{equation}\label{selanumero}\begin{split}&
\langle I'_\varepsilon(v), w\rangle \\&=
\iint_{\R^{2N}}\frac{(v(x)-v(y))(w(x)-w(y))}{|x-y|^{N+2s}}\,dx\,dy
-\int_{\R^N}g(x,v(x))w(x)\,dx\\
&=\iint_{\R^{2N}}\frac{(v(x)-v(y))(w(x)-w(y))}{|x-y|^{N+2s}}\,dx\,dy
\\&\qquad-\int_{\R^N}\Big[ \varepsilon h(x)\Big( (u_\varepsilon+v_+(x))^q-u_\varepsilon^q\Big)
+\Big((u_\varepsilon+v_+(x))^{2^*_s-1}-u_\varepsilon^{2^*_s-1}\Big) \Big]w(x)\,dx.
\end{split}\end{equation}
That is, critical points of~$I_\varepsilon$ are weak solutions of
\begin{equation}\label{nk dcvr9SLDM}
(-\Delta)^s v(x)=\varepsilon h(x)\Big( (u_\varepsilon(x)+ v_+(x))^q-u_\varepsilon^q(x)\Big)
+\Big((u_\varepsilon(x)+v_+(x))^{2^*_s-1}-u_\varepsilon^{2^*_s-1}(x)\Big).
\end{equation}

This can be made more precise, according to the next observation:

\begin{lem}\label{NUOVLEM}
If~$v$ is  a critical point of~$I_\varepsilon$, then~$v\ge0$ and it is a weak solution of
\begin{equation*}
(-\Delta)^s v(x)=\varepsilon h(x)\Big( (u_\varepsilon(x)+ v(x))^q-u_\varepsilon^q(x)\Big)
+\Big((u_\varepsilon(x)+v(x))^{2^*_s-1}-u_\varepsilon^{2^*_s-1}(x)\Big).
\end{equation*}
Also, the function~$u:=u_\varepsilon+v$ is a weak solution of~\eqref{problema} and thus a critical point of~$f_\varepsilon$.
\end{lem}

\begin{proof} We note that, for all~$t\in\R$ and~$x\in\R^n$, we have that~$g(x,t)t_-=0$.
On this account, we choose~$w:=v_-$ as a test function in~\eqref{selanumero} and we find that,
if~$v$ is  a critical point of~$I_\varepsilon$,
\begin{eqnarray*}
0&=&
\iint_{\R^{2N}}\frac{(v(x)-v(y))(v_-(x)-v_-(y))}{|x-y|^{N+2s}}\,dx\,dy
-\int_{\R^N}g(x,v(x))v_-(x)\,dx
\\&=&\iint_{\R^{2N}}\frac{(v(x)-v(y))(v_-(x)-v_-(y))}{|x-y|^{N+2s}}\,dx\,dy\\&\ge&\iint_{\R^{2N}}\frac{(v_-(x)-v_-(y))^2}{|x-y|^{N+2s}}\,dx\,dy,
\end{eqnarray*}yielding that~$v_-$ vanishes identically, hence~$v\ge0$.

The claims about the equations satisfied by~$v$ and~$u$ are now a direct consequence of
the nonnegativity of~$v$ and~$u_\varepsilon$ and~\eqref{nk dcvr9SLDM}.
\end{proof}

In light of Lemma~\ref{NUOVLEM}, from now on we will focus our attention
on the translated functional~$I_\varepsilon$, showing that it satisfies the
hypotheses of the Mountain Pass Theorem under a suitable threshold, provided
that~$v=0$ is the only critical point of~$I_\varepsilon$. This will allow us
to perform a contradiction argument and show the existence of a second
nontrivial critical point (strictly speaking, the second critical point is therefore
not necessarily of mountain pass type, but it is obtained through the application of the Mountain Pass Theorem).

We start our analysis with the following result highlighting the connection between the functional~$f_\varepsilon$ defined in~\eqref{deffepsilon} and~$I_\varepsilon$.

\begin{prop}\label{propvminimolocale}
Let~$u_\varepsilon$ be the nonnegative
local minimum of~$f_\varepsilon$ in~$\mathcal{D}^{s,2}(\R^N)$
given by Proposition~\ref{prop:primasol}. 

Then, $v=0$ is a local minimum of~$I_\varepsilon$ in~$\mathcal{D}^{s,2}(\R^N)$.
\end{prop}

\begin{proof}
Since~$u_\varepsilon$ is a local minimum of~$f_\varepsilon$
in~$\mathcal{D}^{s,2}(\R^N)$, there exists~$\eta>0$ such that,
for all~$[v]_{s}<\eta$, 
we have that \begin{equation}\label{chestrano22}
f_\varepsilon(u_\varepsilon+v)\geq f_\varepsilon(u_\varepsilon).\end{equation}

Also, recalling~\eqref{deffepsilon} and~\eqref{definizioneIepsilon}
(and using also the fact that~$u_\varepsilon$ is nonnegative),
\begin{equation}\label{eqIepsilon}
\begin{split}
&I_\varepsilon(v)-f_\varepsilon(u_\varepsilon+v_+)+f_\varepsilon(u_\varepsilon)\\&=\frac{1}{2}\int_{\R^N}|D^s v(x)|^2\,dx
-\frac{1}{2}\int_{\R^N}|D^s (u_\varepsilon+ v_+)(x)|^2\,dx
+\frac{1}{2}\int_{\R^N}|D^s u_\varepsilon(x)|^2\,dx \\
&\qquad
+\varepsilon \int_{\R^N}h(x)u_\varepsilon^q(x) v_+(x)\,dx
+\int_{\R^N}u_\varepsilon^{2^*_s-1} v_+(x)\,dx.
\end{split}
\end{equation}

We notice that \[
\begin{aligned}
\int_{\R^N}|D^s (u_\varepsilon+ v_+)(x)|^2\,dx&=
\int_{\R^N}|D^s u_\varepsilon(x)|^2\,dx
+\int_{\R^N}|D^s v_+(x)|^2,dx \\
&\qquad+2\iint_{\R^{2N}}\frac{(u_\varepsilon(x)-u_\varepsilon(y))(v_+(x)-v_+(y))}{|x-y|^{N+2s}}\,dx\,dy.
\end{aligned}
\]
Moreover, since~$u_\varepsilon$ is a nonnegative critical point of~$f_\varepsilon$,
\[
\begin{aligned}
0&=\langle f_\varepsilon'(u_\varepsilon),v_+\rangle
\\&=\iint_{\R^{2N}}\frac{(u_\varepsilon(x)-u_\varepsilon(y))(v_+(x)-v_+(y))}{|x-y|^{N+2s}}\,dx\,dy \\&\qquad\qquad
-\varepsilon \int_{\R^N}h(x)u_\varepsilon^q (x)v_+(x)\,dx
-\int_{\R^N}u_\varepsilon^{2^*_s-1}(x) v_+(x)\,dx.
\end{aligned}
\]
Gathering these pieces of information, we deduce from~\eqref{eqIepsilon} that
\begin{equation}\label{eqIepsilon.90}
I_\varepsilon(v)-f_\varepsilon(u_\varepsilon+v_+)+f_\varepsilon(u_\varepsilon)=\frac{1}{2}\int_{\R^N}|D^s v(x)|^2\,dx
-\frac{1}{2}\int_{\R^N}|D^s v_+(x)|^2\,dx.
\end{equation}

Now, since~$|v(x)-v(y)|^2\geq |v_+(x)-v_+(y)|^2$ for any~$x$, $y\in \R^N$, we have that
$$\frac{1}{2}\int_{\R^N}|D^s v(x)|^2\,dx
\ge\frac{1}{2}\int_{\R^N}|D^s v_+(x)|^2\,dx.$$
This and~\eqref{eqIepsilon.90} entail that
$$
I_\varepsilon(v)-f_\varepsilon(u_\varepsilon+v_+)+f_\varepsilon(u_\varepsilon)\ge0.$$

Accordingly, from this and~\eqref{chestrano22}
we conclude that, if~$[v_+]_s\le[v]_{s}<\eta$,
\[
I_\varepsilon(v)\geq f_\varepsilon(u_\varepsilon+v_+)-f_\varepsilon(u_\varepsilon)\ge 0=I_\varepsilon(0).
\]
Hence~$0$ is a local minimum of~$I_\varepsilon$, as desired.
\end{proof}

Now we prove that if~$v=0$ is the only critical point of~$I_\varepsilon$, then~$I_\varepsilon$ satisfies the Palais-Smale
condition below a suitable threshold.
To this end, we present the following result.

\begin{lem}\label{lemmavnbounded}
Suppose that~$(h_0)$ and~$(h_1)$ hold true and let~$\{v_n\}_n \subset \mathcal{D}^{s,2}(\R^N)$ be such that
\begin{equation}\label{PSprimaIe-fi01}
\displaystyle{\lim_{n\to +\infty}}I_\varepsilon(v_n)=c<\frac{s}{N}S^\frac{N}{2s}\end{equation}
and
\begin{equation}\label{PSprimaIe-fi02}\displaystyle{\lim_{n\to +\infty}}\sup_{{w\in\mathcal{D}^{s,2}(\R^N)}\atop{[w]_s\le1}}
\big|\langle I'_\varepsilon(v_n),w\rangle\big|= 0.\end{equation}

Then, the sequence~$\{v_n\}_n$ is bounded in~$\mathcal{D}^{s,2}(\R^N)$.
\end{lem}

\begin{proof}
We first show that, for all~$\delta>0$, there exists~$C_{\delta}>0$ such that
\begin{equation}\label{claim6.2.3}
\int_{\R^N}\left(G(x,v_n(x))-\frac{1}{2^*_s}g(x,v_n(x))v_n(x) \right)\,dx
\leq C(\varepsilon+\delta)\|(v_n)_+\|_{L^{2^*_s}(\R^N)}^{2^*_s}+C\varepsilon+C_{\delta}
\end{equation}
for some~$C>0$ independent of~$\varepsilon$ and~$\delta$.

To prove this, recalling the definition of~$I_\varepsilon$ in~\eqref{definizioneIepsilon}, we write~$g=g_1+g_2$ and~$G=G_1+G_2$, where
\[
\begin{aligned}
&g_1(x,t):=\varepsilon h(x)\Big((u_\varepsilon(x)+t_+)^q-u_\varepsilon^q(x)\Big),
\\
&g_2(x,t):=(u_\varepsilon(x)+t_+)^{2^*_s-1}-u_\varepsilon^{2^*_s-1}(x),
\\
&G_1(x,t):=\frac{\varepsilon h(x)}{q+1}
\Big((u_\varepsilon(x)+t_+)^{q+1}-u_\varepsilon^{q+1}(x)\Big)
-\varepsilon h(x)u_\varepsilon^q (x)t_+
\\{\mbox{and }}\qquad
&G_2(x,t):=\frac{1}{2^*_s}
\Big((u_\varepsilon(x)+t_+)^{2^*_s}-u_\varepsilon^{2^*_s}(x)\Big)
-u_\varepsilon^{2^*_s-1}(x) t_+.
\end{aligned}
\]
Since~$q\in (0,1)$, we have that, for any~$\tau \geq 0$,
\begin{equation}\label{1234567890098765432qwertyuio}
(1+\tau)^q-1=q\int_0^\tau (1+\theta)^{q-1}\,d\theta
\leq q\int_0^\tau \theta^{q-1}\,d\theta =\tau^q
\end{equation}
and
\begin{equation}\label{1234567890098765432qwertyuio2}
(1+\tau)^{q+1}-1=(q+1)\int_0^\tau (1+\theta)^q\,d\theta
\leq (q+1)(1+\tau)^q\tau.
\end{equation} 

We observe that, for all~$x\in\R^N$,
\begin{equation}\label{equa:01901}\begin{split}
&(u_\varepsilon(x)+t_+)^q-u_\varepsilon^q(x)
\le t_+^q
\\{\mbox{and }}\qquad & 
(u_\varepsilon(x)+t_+)^{q+1}-u_\varepsilon^{q+1}(x)
\le (q+1)(u_\varepsilon(x)+t_+)^qt_+.\end{split}
\end{equation}
Indeed, if~$u_\varepsilon(x)=0$, then the inequalities in~\eqref{equa:01901}
are trivially satisfied. If instead~$u_\varepsilon(x)>0$,
we take~$\tau:=t_+/u_\varepsilon( x)$ 
in~\eqref{1234567890098765432qwertyuio}
and~\eqref{1234567890098765432qwertyuio2}
and
we find that
\begin{eqnarray*}
(u_\varepsilon(x)+t_+)^q-u_\varepsilon^q(x)
=u_\varepsilon^q(x)\big((1+\tau)^q-1\big)
\leq u_\varepsilon^q(x) \tau^q=t_+^q\end{eqnarray*}
and \begin{eqnarray*} &&
(u_\varepsilon(x)+t_+)^{q+1}-u_\varepsilon^{q+1}(x)
=u_\varepsilon^{q+1}(x)\big((1+\tau)^{q+1}-1\big) \\&&\qquad\qquad
\leq (q+1)u_\varepsilon^{q+1}(x)(1+\tau)^q\tau
=(q+1)(u_\varepsilon(x)+t_+)^qt_+,
\end{eqnarray*}
which give the desired inequalities in~\eqref{equa:01901}.

As a consequence of~\eqref{equa:01901}, we have that
\begin{eqnarray*}&&
|g_1(x,t)|\leq \varepsilon |h(x)|t_+^q
\\{\mbox{and }}
&&
|G_1(x,t)|\leq \varepsilon |h(x)|(u_\varepsilon(x)+t_+)^q t_+
+\varepsilon |h(x)|u_\varepsilon^q (x)t_+
\leq 2 \varepsilon |h(x)|(u_\varepsilon(x)+t_+)^q t_+,
\end{eqnarray*}
which yield that
\begin{equation}\label{GOMI}
|G_1(x,t)|+|g_1(x,t)t|
\leq 2 \varepsilon |h(x)|(u_\varepsilon(x)+t_+)^q t_+
+\varepsilon |h(x)|t_+^{q+1}
\leq 3 \varepsilon |h(x)|(u_\varepsilon(x)+t_+)^q t_+.
\end{equation}

We also recall that~$u_\varepsilon\in L^\infty(\R^N)$ 
and~$u_\varepsilon\to0$ in~$\mathcal{D}^{s,2}(\R^N)$ as~$\varepsilon\to0$,
thanks to Proposition~\ref{prop:primasol}, which entails that~$\|u_\varepsilon\|_{L^\infty(\R^N)}$ is bounded uniformly in~$\varepsilon$
(see e.g. formula~(5.1.6)
and the last formula in the proof of
Proposition~5.1.1 in~\cite{MR3617721}).
For this reason, it follows from~\eqref{GOMI} that
\begin{equation}\label{equazione6.2.5}
|G_1(x,t)|+|g_1(x,t)t|\leq C \varepsilon |h(x)|(1+t_+^{q+1})
\end{equation}
for some~$C>0$ independent of~$\varepsilon$. 

Additionally, since~$q+1<2^*_s$, we have that~$
t_+^{q+1}\leq 1+t_+^{2^*_s}$,
and therefore~\eqref{equazione6.2.5} entails that
\begin{equation}\label{equazione6.2.6}
|G_1(x,t)|+|g_1(x,t)t|\leq C \varepsilon |h(x)|(1+t_+^{2^*_s})
\end{equation}
up to changing~$C$.

Moreover, by Lemma~6.2.1 in~\cite{MR3617721}, for any~$\delta>0$ there 
exists~$M_\delta>0$ such that, for any~$\alpha$, $\beta \geq 0$ 
and~$m>0$,
\begin{equation*}
(\alpha+\beta)^{m+1}-\alpha^{m+1}-(m+1)\alpha^m\beta
-\beta((\alpha+\beta)^{m}-\alpha^{m})
\leq \delta \beta^{m+1}+M_\delta \alpha^{m+1}.
\end{equation*}
Now, given~$\delta \in(0,1)$, we use this inequality
with~$\alpha: =u_\varepsilon$, $\beta:=t_+$ and~$m:=2^*_s-1$ to see that
\[
\begin{aligned}&
G_2(x,t)-\frac{1}{2^*_s}g_2(x,t)t\\
&=\frac{1}{2^*_s}\Big((u_\varepsilon(x)+t_+)^{2^*_s}-u_\varepsilon^{2^*_s}(x)
-2^*_s u_\varepsilon^{2^*_s-1}(x)t_+
-(u_\varepsilon(x)+t_+)^{2^*_s-1}t+u_\varepsilon^{2^*_s-1}(x)t\Big)  \\
&=\frac{1}{2^*_s}\Big((u_\varepsilon(x)+t_+)^{2^*_s}-u_\varepsilon^{2^*_s}(x)
-2^*_s u_\varepsilon^{2^*_s-1}(x)t_+
-(u_\varepsilon(x)+t_+)^{2^*_s-1}t_+ +u_\varepsilon^{2^*_s-1}(x)t_+\Big) \\
&\leq \frac{1}{2^*_s}\big(\delta t_+^{2^*_s}+M_\delta 
u_\varepsilon^{2^*_s}(x)\big).
\end{aligned}
\]
Together with~\eqref{equazione6.2.6}, this gives that
\[
\begin{aligned}
G(x,t)-\frac{1}{2^*_s}g(x,t)t
&=G_1(x,t)-\frac{1}{2^*_s}g_1(x,t)t 
+G_2(x,t)-\frac{1}{2^*_s}g_2(x,t)t \\
&\leq C\varepsilon |h(x)|(1+t_+^{2^*_s})
+\frac{1}{2^*_s}\big(\delta t_+^{2^*_s}+M_\delta 
u_\varepsilon^{2^*_s}(x)\big) \\
&\leq C(\varepsilon +\delta)t_+^{2^*_s} +C\varepsilon|h(x)|
+M_\delta u_\varepsilon^{2^*_s}(x),
\end{aligned}
\]
up to changing the constants~$C$ and~$M_\delta$. 

As a consequence,
\begin{eqnarray*}
&&\int_{\R^N}\left(G(x,v_n(x))-\frac{1}{2^*_s}g(x,v_n(x))v_n(x) \right)\,dx\\
&&\qquad\le C(\varepsilon +\delta)\int_{\R^N}
(v_n)_+^{2^*_s}(x)\,dx +C\varepsilon\int_{\R^N}|h(x)|\,dx
+M_\delta \int_{\R^N}u_\varepsilon^{2^*_s}(x)\,dx\\&&\qquad
\leq C(\varepsilon +\delta)\|(v_n)_+\|_{L^{2^*_s}(\R^N)}^{2^*_s}
+C\varepsilon+M_\delta
\end{eqnarray*}
up to renaming~$C$ and~$M_\delta$ once more, and therefore the claim in~\eqref{claim6.2.3} is established.

Now we show that 
\begin{equation}\label{claim6.2.4}
\int_{\R^N}g(x,v_n(x))v_n(x)\,dx
\geq \frac{\|(v_n)_+\|_{L^{2^*_s}(\R^N)}^{2^*_s}}{8}-\widehat C
\end{equation}
for some~$\widehat C>0$. 

To prove this claim, for any~$\tau \geq 0$, we set
\[
\ell(\tau):=\frac{\tau^{2^*_s-1}}{2}-(1+\tau)^{2^*_s-1}+1.
\]
We observe that~$\ell(0)=0$ and 
\[
\lim_{\tau \to +\infty}\ell(\tau)=-\infty.
\]
Therefore,
\[
L:=\sup_{\tau \geq 0}\ell(\tau)\in [0,+\infty),
\]
and thus
\begin{equation}\label{nbvcx123456789}
(1+\tau)^{2^*_s-1}-1=\frac{\tau^{2^*_s-1}}{2}-\ell(\tau)
\geq \frac{\tau^{2^*_s-1}}{2}-L.
\end{equation}

We point out that, for all~$x\in\R^N$,
\begin{equation}\label{equa:019012}
g_2(x,v_n(x))\ge\frac{(v_n)_+^{2^*_s-1}(x)}{2}-Lu_\varepsilon^{2^*_s-1}(x).
\end{equation}
Indeed, if~$u_\varepsilon(x)=0$, then~$g_2(x,v_n(x))=(v_n)_+^{2^*_s-1}(x)$, which entails~\eqref{equa:019012}. If insted~$u_\varepsilon(x)\neq0$,
then we can take~$\tau:=(v_n)_+(x)/u_\varepsilon(x)$
in~\eqref{nbvcx123456789} and obtain that
\begin{equation*}
\begin{aligned}
g_2(x,v_n(x))&=(u_\varepsilon(x)+(v_n)_+(x))^{2^*_s-1}-u_\varepsilon^{2^*_s-1}(x)
=u_\varepsilon^{2^*_s-1}(x)\big((1+\tau)^{2^*_s-1}-1\big) \\
&\qquad\geq u_\varepsilon^{2^*_s-1}(x)
\left( \frac{\tau^{2^*_s-1}}{2}-L\right)
=\frac{(v_n)_+^{2^*_s-1}(x)}{2}-Lu_\varepsilon^{2^*_s-1}(x),
\end{aligned}
\end{equation*}
thus completing the proof of~\eqref{equa:019012}.

Hence, integrating~\eqref{equa:019012} and using the Young inequality, we obtain that
\begin{equation}\label{equazione6.2.7}\begin{split}&
\int_{\R^N}g_2(x,v_n(x))v_n(x) \,dx=
\int_{\R^N}g_2(x,v_n(x))(v_n)_+(x) \,dx
\\&\qquad\ge
\int_{\R^N}\left(
\frac{(v_n)_+^{2^*_s}(x)}{2}-L
u_\varepsilon^{2^*_s-1}(x)(v_n)_+(x)\right)\,dx
\geq \frac{\|(v_n)_+\|_{L^{2^*_s}(\R^N)}^{2^*_s}}{4}-\widetilde C
\end{split}
\end{equation}
for some~$\widetilde C>0$ independent of~$\varepsilon$.

Additionally, by~\eqref{equazione6.2.6} we have that
\[
\left| \int_{\R^N}g_1(x,v_n(x))v_n(x) \,dx \right|
\leq C\varepsilon\int_{\R^N}|h(x)|\big(1+(v_n)_+^{2^*_s}(x)\big)\,dx 
\leq C\varepsilon\Big(1+ \|(v_n)_+\|_{L^{2^*_s}(\R^N)}^{2^*_s}\Big),
\]
up to renaming~$C$.

Combining this with~\eqref{equazione6.2.7}, we obtain that, if~$\varepsilon$
is sufficiently small,
\begin{eqnarray*}
\int_{\R^N}g(x,v_n(x))v_n(x) \,dx &=&
\int_{\R^N}g_1(x,v_n(x))v_n(x) \,dx +\int_{\R^N}g_2(x,v_n(x))v_n(x) \,dx  \\
& \geq& \frac{\|(v_n)_+\|_{L^{2^*_s}(\R^N)}^{2^*_s}}{4}-\widetilde C
-C\varepsilon\Big(1+ \|(v_n)_+\|_{L^{2^*_s}(\R^N)}^{2^*_s}\Big)\\
&
\geq& \frac{\|(v_n)_+\|_{2^*_s}^{2^*_s}}{8}-\widehat C,
\end{eqnarray*}
for some~$\widehat C>0$.
This establishes the claim in~\eqref{claim6.2.4}.

Now, we show that the sequence~$\{v_n\}_n$ is bounded in~$\mathcal{D}^{s,2}(\R^N)$. For this, let~$\lambda>1$.
{F}rom~\eqref{PSprimaIe-fi01} and~\eqref{PSprimaIe-fi02} 
we have that
\begin{equation}\label{PSprimaIe}
I_\varepsilon(v_n)-\frac{1}{\lambda}\langle I_\varepsilon'(v_n),v_n \rangle
\le M+\frac{M}{\lambda}[v_n]_{s},\end{equation} for some~$M>0$.

Exploiting this inequality with~$\lambda:=2$, we get that
\begin{eqnarray*}
&&M+\frac{M}{2}[v_n]_{s}\\&
\geq&
I_\varepsilon(v_n)-\frac{1}{2}\langle I_\varepsilon'(v_n),v_n \rangle
\\&=&-\int_{\R^N}G(x,v_n(x)) \,dx +\frac{1}{2}\int_{\R^N}g(x,v_n(x))v_n(x) \,dx
\\
&=&\frac{1}{2^*_s}\int_{\R^N}g(x,v_n(x))v_n(x) \,dx
-\int_{\R^N}G(x,v_n(x))\,dx
+\left(\frac{1}{2}-\frac{1}{2^*_s}\right)\int_{\R^N}g(x,v_n(x))v_n(x) \,dx
\\
&=&\frac{1}{2^*_s}\int_{\R^N}g(x,v_n(x))v_n(x) \,dx
-\int_{\R^N}G(x,v_n(x))\,dx
+\frac{s}{N}\int_{\R^N}g(x,v_n(x))v_n(x) \,dx.
\end{eqnarray*}
Accordingly, given~$\delta>0$ we use~\eqref{claim6.2.3} and~\eqref{claim6.2.4}
and we thereby arrive at
\[
M+\frac{M}{2}[v_n]_{s}
\geq -C(\varepsilon+\delta)\|(v_n)_+\|^{2^*_s}_{L^{2^*_s}(\R^N)} 
-C\varepsilon -C_{\delta}
+\frac{s}{N}\left(\frac{\|(v_n)_+\|_{L^{2^*_s}(\R^N)}^{2^*_s}}{8}-\widehat C \right).
\]
We now take~$\delta:=\frac{s}{16CN}$ (which is now fixed once and for all)
and we get that
\[
M+\frac{M}{2}[v_n]_{s}
\geq -C\varepsilon\|(v_n)_+\|^{2^*_s}_{L^{2^*_s}(\R^N)} 
-C{\varepsilon}-C
+\frac{s}{N}\left(\frac{\|(v_n)_+\|_{L^{2^*_s}(\R^N)}^{2^*_s}}{16}-\widehat C \right).
\]
We also take~$\varepsilon$ sufficiently smally, say~$\varepsilon\in\left(0,\frac{s}{32CN}\right)$, and we obtain that
\[
M+\frac{M}{2}[v_n]_{s}
\geq \frac{s}{N}\frac{\|(v_n)_+\|_{L^{2^*_s}(\R^N)}^{2^*_s}}{32}
-C,
\]
up to renaming the constant~$C>0$. 

We can write this inequality as
\begin{equation}\label{equazione6.2.9}
\|(v_n)_+\|_{L^{2^*_s}(\R^N)}^{2^*_s}\leq \widetilde M \big([v_n]_{s}+1\big)
\end{equation}
for some~$\widetilde M>0$.

Furthermore, employing~\eqref{claim6.2.3} with~$\delta:=1$
and~\eqref{PSprimaIe} with~$\lambda:=2^*_s$, we gather that
\[
\begin{aligned}
M+\frac{M}{2^*_s}[v_n]_{s}
&\geq I_\varepsilon(v_n)-\frac{1}{2^*_s}\langle I_\varepsilon'(v_n),(v_n)\rangle \\
&=\left(\frac{1}{2}-\frac{1}{2^*_s}\right)[v_n]_{s}^2
+\frac{1}{2^*_s}\int_{\R^N}g(x,v_n(x))v_n(x) \,dx-\int_{\R^N}G(x,v_n(x))\,dx
\\
&\geq \frac{s}{N}[v_n]_{s}^2-C\Big(\|(v_n)_+\|_{2^*_s}^{2^*_s}+1\Big)
\end{aligned}
\]
provided that~$\varepsilon$ is sufficiently small. 

As a consequence,
\[
[v_n]_{s}^2
\leq \widehat M\Big([v_n]_{s}+\|(v_n)_+\|_{L^{2^*_s}(\R^N)}^{2^*_s}+1 \Big)
\]
for some~$\widehat M>0$.

{F}rom this and~\eqref{equazione6.2.9}, we infer that
\[
[v_n]_{s}^2 \leq M\big([v_n]_{s}+1 \big),
\]
for some~$M>0$.
Thus~$\{v_n\}_n$ is bounded in~$\mathcal{D}^{s,2}(\R^N)$, as desired.
\end{proof}

We now prove that the functional~$I_\varepsilon$ satisfies the 
Palais-Smale condition below a suitable level~$c$.
The strategy of the proof is similar 
to that of Proposition~\ref{PSfepsilon}, nevertheless we provide
here below a self-contained proof,
as in this case there are some additional difficulties to take care of. 

\begin{prop}\label{PSIepsilon}
Suppose that~$(h_0)$ and~$(h_1)$ hold true.
Assume that~$v=0$ is the only critical point of~$I_\varepsilon$ in~$\mathcal{D}^{s,2}(\R^N)$.

Let~$\{v_n\}_n \subset \mathcal{D}^{s,2}(\R^N)$ be such that
\begin{equation}\label{CCM}
\displaystyle{\lim_{n\to +\infty}}I_\varepsilon(v_n)=c<\frac{s}{N}S^\frac{N}{2s}\end{equation}
and \begin{equation}\label{CCMBIS}
{\lim_{n\to +\infty}}\sup_{{w\in\mathcal{D}^{s,2}(\R^N)}\atop{[w]_s\le1}}\big|\langle I'_\varepsilon(v_n), w\rangle\big|= 0.
\end{equation}

Then, up to a subsequence, $\{v_n\}_n$ strongly converges to~$v=0$ in~$\mathcal{D}^{s,2}(\R^N)$ as~$n\to+\infty$.
\end{prop}

\begin{proof}
Thanks to Lemma~\ref{lemmavnbounded}, we know that~$\{v_n\}_n$ is bounded in~$\mathcal{D}^{s,2}(\R^N)$, and thus so is~$\{(v_n)_+\}_n$.
Therefore, up to a subsequence,
\begin{equation*} \begin{split}
&v_n \rightharpoonup v_\infty \mbox{ weakly in }\mathcal{D}^{s,2}(\R^N),
\\ &v_n\rightarrow v_\infty \mbox{ in } L^r_{loc}({\R^N}) \;{\mbox{ for every~$r\in [1,2^*_s)$,}}
\\&
v_n\to v_\infty \mbox{ a.e. in } \R^N,\\
&(v_n)_+ \rightharpoonup (v_\infty)_+ \mbox{ weakly in }\mathcal{D}^{s,2}(\R^N),
\\ &(v_n)_+\rightarrow (v_\infty)_+ \mbox{ in } L^r_{loc}({\R^N}) \;{\mbox{ for every~$r\in [1,2^*_s)$}}
\\ \mbox{and}\qquad&
(v_n)_+\to (v_\infty)_+ \mbox{ a.e. in } \R^N.
\end{split}\end{equation*}
{F}rom the strong convergence of~$(v_n)_+$, we can assume up to a 
subsequence that~$u_\varepsilon+(v_n)_+$ converges to~$u_\varepsilon+(v_\infty)_+$ in~$L^r_{loc}(\R^N)$ for every~$r\in [1, 2^*_s)$.

We point out that
\begin{equation}\label{fhuiet43yt43ty893098765}
{\mbox{$v_\infty$ is a critical point of~$I_\varepsilon$.}}\end{equation}
To check this, one can argue exactly as in the proof of Lemma~\ref{wea},
with suitable modifications for the convergence of the following terms:
given~$w\in C^\infty_c(\R^N)$ with supp$(\phi)\subset K$,
we can use the H\"older inequality and
Lemma~\ref{lemmmaconvergenzaLrloc} (applied with~$u_n:=(u_\varepsilon+(v_n)_+)$) to see that
\[
\begin{aligned}
\lim_{n\to +\infty} 
&\left|\int_{\R^N}h(x)\big|\big(
u_\varepsilon(x)+(v_n)_+(x)\big)^q-\big(u_\varepsilon(x)+(v_\infty)_+(x)\big)^q\big|\,w(x)\,dx\right| \\
&\leq \lim_{n\to +\infty} C
\int_{K} \big|\big(u_\varepsilon(x)+(v_n)_+(x)\big)^q-\big(u_\varepsilon(x)+(v_\infty)_+(x)\big)^q\big|^\frac{r}{q}\,dx
=0
\end{aligned}
\]
and
\[
\begin{aligned}
\lim_{n\to +\infty} 
&\left|\int_{\R^N}\big|\big(u_\varepsilon(x)+(v_n)_+(x)\big)^{2^*_s-1}
-\big(u_\varepsilon(x)+(v_\infty)_+(x)\big)^{2^*_s-1}\big|\,w(x)\,dx \right|\\
&\leq \lim_{n\to +\infty} C
\int_{K} \big|\big(u_\varepsilon(x)+(v_n)_+(x)\big)^{2^*_s-1}-\big(u_\varepsilon(x)+(v_\infty)_+(x)\big)^{2^*_s-1}\big|^\frac{r}{2^*_s-1}\,dx=0.
\end{aligned}
\]

Having established~\eqref{fhuiet43yt43ty893098765},
since by hypothesis we have that~$v=0$ is the only 
critical point of~$I_\varepsilon$, we conclude that~$v_\infty=0$.
In particular, it follows that~$u_\varepsilon+ (v_n)_+ \rightharpoonup u_\varepsilon$ weakly in~$\mathcal{D}^{s,2}(\R^N)$.

We can thus apply Theorem~1.1 in~\cite{MR3866572} to find that 
there exist two bounded measures~$\mu$ and~$\nu$, an at most 
countable set of indices~$I$, and positive real numbers~$\mu_i$ and~$\nu_i$, with~$i\in I$, such that the following convergence 
holds weakly* in the sense of measures:
\begin{align}\label{CCvmui}
&|D^s (u_\varepsilon+ (v_n)_+)|^2\,dx \rightharpoonup \mu
\geq |D^s u_\varepsilon|^2\,dx
+\sum_{i\in I}\mu_i \delta_{x_i}\\ \label{CCvnui}
{\mbox{and }}\qquad &|u_\varepsilon+ (v_n)_+|^{2^*_s}\,dx \rightharpoonup \nu
= |u_\varepsilon|^{2^*_s}\,dx
+\sum_{i\in I}\nu_i \delta_{x_i}.
\end{align}
Moreover, for all~$ i \in I$,
\begin{equation}\label{CCvSi}
S^\frac{1}{2}\nu_i^\frac{1}{2^*_s}\leq \mu_i^\frac{1}{2}.
\end{equation}

Furthermore, Theorem~1.1 in~\cite{MR3866572} also gives that, defining
\begin{eqnarray*}
\mu_\infty&:=&\lim_{R\to +\infty}\limsup_{n\to +\infty}
\int_{\R^N\setminus B_R}|D^s(u_\varepsilon+(v_n)_+)(x)|^2\,dx\\{\mbox{and }}\qquad
\nu_\infty&:=&\lim_{R\to +\infty}\limsup_{n\to +\infty}
\int_{\R^N\setminus B_R}|u_\varepsilon(x)+(v_n)_+(x)|^{2^*_s}\,dx,
\end{eqnarray*}
we get that
\begin{align}\label{CCvmuinfinito}
&\limsup_{n\to +\infty}
\int_{\R^N}|D^s(u_\varepsilon+ (v_n)_+)(x)|^2\,dx=\mu(\R^N)+\mu_\infty,
\\{\mbox{and }}\qquad& \label{CCvnuinfinito}
\limsup_{n\to +\infty}
\int_{\R^N}|u_\varepsilon(x)+(v_n)_+(x)|^{2^*_s}\,dx=\nu(\R^N)+\nu_\infty,
\\
\label{CCvSinfinito} {\mbox{with }}\qquad&
S^\frac{1}{2}\nu_\infty^\frac{1}{2^*_s}\leq \mu_\infty^\frac{1}{2}.
\end{align}

Now, we claim that 
\begin{equation}\label{claimvmunu} {\mbox{if~$\mu_i\neq0$ for some~$i\in I$, then~$\nu_i\ge S^{\frac{N}{2s}}$.}}
\end{equation}
To check this, we make use
of~\eqref{selanumero0} and~\eqref{selanumero} to perform the following
preliminary computation:
\begin{equation}\label{xherg56cjhiamala43t}\begin{split}
&\langle I_\varepsilon'(v_n),w \rangle  +\langle f_\varepsilon'(u_\varepsilon),w \rangle \\&=
\iint_{\R^{2N}}\frac{(v_n(x)-v_n(y))(w(x)-w(y))}{|x-y|^{N+2s}}\,dx\,dy
\\&\qquad-\int_{\R^N}\Big[ \varepsilon h(x)\Big( (u_\varepsilon(x)+(v_n)_+(x))^q-u_\varepsilon^q\Big)
+\Big((u_\varepsilon(x)+(v_n)_+(x))^{2^*_s-1}-u_\varepsilon^{2^*_s-1}\Big) \Big]w(x)\,dx \\&\qquad+
\iint_{\R^{2N}}\frac{(u_\varepsilon(x)-u_\varepsilon(y))(w(x)-w(y))}{|x-y|^{N+2s}}\,dx\,dy\\&\qquad
-\varepsilon \int_{\R^N}h(x)u_\varepsilon^{q}(x)w(x)\,dx
-\int_{\R^N}u_\varepsilon^{2^*_s-1}(x)w(x)\,dx
\\&= \iint_{\R^{2N}}\frac{\big(u_\varepsilon(x)+v_n(x)-u_\varepsilon(y)-v_n(y)\big)(w(x)-w(y))}{|x-y|^{N+2s}}\,dx\,dy\\
&\qquad
-\int_{\R^N}\varepsilon h(x)\big(u_\varepsilon(x)+(v_n)_+(x)\big)^qw(x)\,dx
-\int_{\R^N}\big(u_\varepsilon(x)+(v_n)_+(x)\big)^{2^*_s-1}w(x)\,dx.
\end{split}\end{equation}
We now exploit this formula with~$w:=(u_\varepsilon+v_n)\phi_\delta$,
where~$\phi_\delta:=\phi_{\delta,x_i}$ for some~$\delta>0$
(recall~\eqref{definizionephi}). In this way, since~$u_\varepsilon$ is a critical point for~$f_\varepsilon$, we conclude that
\begin{equation}\label{scontrovdelta}
\begin{aligned}
0&=\lim_{n\to +\infty} 
\Big(\langle I_\varepsilon'(v_n),(u_\varepsilon+v_n)\phi_\delta \rangle  +\langle f_\varepsilon'(u_\varepsilon),(u_\varepsilon+v_n)\phi_\delta \rangle \Big) \\
&=\lim_{n\to +\infty} 
\Bigg( 
\iint_{\R^{2N}}(u_\varepsilon(y)+v_n(y))
\frac{(u_\varepsilon(x)+v_n(x)-u_\varepsilon(y)-v_n(y))(\phi_\delta(x)-\phi_\delta(y))}{|x-y|^{N+2s}}\,dx\,dy \\
&+\int_{\R^N}\phi_\delta|D^s(u_\varepsilon+v_n)|^2\,dx
-\varepsilon\int_{\R^N}h(u_\varepsilon+(v_n)_+)^{q+1}\phi_\delta\,dx
+\varepsilon\int_{\R^N}h(u_\varepsilon+(v_n)_+)^q(v_n)_-\phi_\delta\,dx \\
&-\int_{\R^N}(u_\varepsilon+(v_n)_+)^{2^*_s}\phi_\delta\,dx
+\int_{\R^N}(u_\varepsilon+(v_n)_+)^{2^*_s-1}(v_n)_-\phi_\delta\,dx
\Bigg).
\end{aligned}
\end{equation}

{F}rom~\eqref{CCvmui} and~\eqref{CCvnui} we know that 
\begin{align}\label{limitephideltaDu+v}&
\lim_{n\to+\infty}
\int_{\R^N}\phi_\delta|D^s(u_\varepsilon+(v_n)_+)|^2\,dx
=\int_{\R^N}\phi_\delta\,d\mu \\
\label{limitephideltau+v}
{\mbox{and }}\quad&\lim_{n\to+\infty}
\int_{\R^N}\phi_\delta|u_\varepsilon+(v_n)_+|^{2^*_s}\,dx
=\int_{\R^N}\phi_\delta\,d\nu.
\end{align}

Moreover, since supp$(\phi_\delta)$ is bounded, up to a subsequence we have that
\begin{equation*}
\lim_{n\to+\infty}
\int_{\R^N}h(u_\varepsilon+(v_n)_+)^{q+1}\phi_\delta\,dx
=\int_{B_\delta(x_i)}h u_\varepsilon^{q+1}\phi_\delta\,dx,
\end{equation*}
and therefore
\begin{equation}\label{limitehphideltau+v}
\lim_{\delta\to0}\lim_{n\to+\infty}
\int_{\R^N}h(u_\varepsilon+(v_n)_+)^{q+1}\phi_\delta\,dx
=0.
\end{equation}

Also, 
\begin{equation}\label{limitev+v-1}\begin{split}
&\lim_{n\to+\infty}
\int_{\R^N}h(u_\varepsilon+(v_n)_+)^q(v_n)_-\phi_\delta\,dx
=0\\ {\mbox{and }}\quad &
\lim_{n\to+\infty}
\int_{\R^N}(u_\varepsilon+(v_n)_+)^{2^*_s-1}(v_n)_-\phi_\delta\,dx
=0.\end{split}
\end{equation}

Furthermore, by the H\"older inequality,
\[
\begin{aligned}
\iint_{\R^{2N}}(u_\varepsilon(y)+v_n(y))&
\frac{(u_\varepsilon(x)+v_n(x)-u_\varepsilon(y)-v_n(y))(\phi_\delta(x)-\phi_\delta(y))}{|x-y|^{N+2s}}\,dx\,dy \\
&\leq 4[u_\varepsilon+v_n]_{s}
\left(\int_{\R^N}|u_\varepsilon+v_n|^2|D^s\phi_\delta|^2\,dx
\right)^\frac{1}{2} \\
&\leq M\left(\int_{\R^N}|u_\varepsilon+v_n|^2|D^s\phi_\delta|^2\,dx
\right)^\frac{1}{2},
\end{aligned}
\] for some~$M>0$,
and the latter quantity converges to
$$ M\left(\int_{\R^N}|u_\varepsilon|^2|D^s\phi_\delta|^2\,dx
\right)^\frac{1}{2}$$
as~$n\to+\infty$.

Additionally, from Lemma~\ref{lemmalimiteDphidelta},
\begin{equation*}
\lim_{\delta \to 0}\int_{\R^N}|u_\varepsilon|^2|D^s\phi_\delta|^2\,dx=0
\end{equation*}
and consequently
\begin{equation*}\lim_{\delta \to 0}\lim_{n\to +\infty}\iint_{\R^{2N}}(u_\varepsilon(y)+v_n(y))
\frac{(u_\varepsilon(x)+v_n(x)-u_\varepsilon(y)-v_n(y))(\phi_\delta(x)-\phi_\delta(y))}{|x-y|^{N+2s}}\,dx\,dy=0.
\end{equation*}
Therefore, combining this result with the ones in~\eqref{scontrovdelta}, 
\eqref{limitephideltau+v}, \eqref{limitehphideltau+v}
and~\eqref{limitev+v-1}, we conclude that
\begin{equation}\label{questainfor7543786qui}\begin{split}
0&=\lim_{\delta \to 0}\lim_{n\to +\infty} 
\Big(\langle I_\varepsilon'(v_n),(u_\varepsilon+v_n)\phi_\delta \rangle  +\langle f_\varepsilon'(v_n),(u_\varepsilon+v_n)\phi_\delta \rangle \Big)\\&=
\lim_{\delta\to0}\left(\lim_{n\to +\infty}
\int_{\R^N}\phi_\delta|D^s(u_\varepsilon+v_n)|^2\,dx
-\int_{\R^N}\phi_\delta\,d\nu\right).
\end{split}\end{equation}

Now, for all~$x$, $h\in\R^N$, we use the short notation
\begin{eqnarray*} &&U_h(x):=u_\varepsilon(x+h)-u_\varepsilon(x),\qquad
V_h(x):=v_n(x+h)-v_n(x)\\&&{\mbox{and}}\qquad
W_h(x):=(v_n)_+(x+h)-(v_n)_+(x)\end{eqnarray*}
and we observe that
\begin{eqnarray*}
&&|D^s(u_\varepsilon+v_n)(x)|^2 -|D^s(u_\varepsilon+(v_n)_+)(x)|^2
\\&=&\int_{\R^N}\frac{|(u_\varepsilon+v_n)(x+h)-(u_\varepsilon+v_n)(x)|^2
-|(u_\varepsilon+(v_n)_+)(x+h)-(u_\varepsilon+(v_n)_+)(x)|^2}{|h|^{N+2s}}\,dh\\&=&
\int_{\R^N}\frac{|U_h(x)+V_h(x)|^2
-|U_h(x)+W_h(x)|^2}{|h|^{N+2s}}\,dh\\&=&
\int_{\R^N}\frac{2U_h(x)\big(V_h(x)-W_h(x)\big)
+|V_h(x)|^2-|W_h(x)|^2}{|h|^{N+2s}}\,dh\\&\ge&
\int_{\R^N}\frac{2U_h(x)\big(V_h(x)-W_h(x)\big)}{|h|^{N+2s}}\,dh\\&=&
\int_{\R^N}\frac{2\big(u_\varepsilon(x+h)-u_\varepsilon(x)\big)
\big( v_n(x+h)-v_n(x) \big)}{|h|^{N+2s}}\,dh
\\&&\qquad-\int_{\R^N}\frac{2\big(u_\varepsilon(x+h)-u_\varepsilon(x)\big)
\big( (v_n)_+(x+h)-(v_n)_+(x) \big)}{|h|^{N+2s}}\,dh.
\end{eqnarray*}
The weak convergence of~$v_n$ and~$(v_n)_+$ to~$0$ in~$\mathcal{D}^{s,2}(\R^N)$ thus gives that
\begin{eqnarray*}&&
\lim_{n\to +\infty}\left(
\int_{\R^N}\phi_\delta|D^s(u_\varepsilon+v_n)|^2\,dx-
\int_{\R^N}\phi_\delta|D^s(u_\varepsilon+(v_n)_+)|^2\,dx\right)\\&\ge&
2\lim_{n\to +\infty}\Bigg[
\iint_{\R^{2N}}\frac{\big(u_\varepsilon(x+h)-u_\varepsilon(x)\big)
\big( v_n(x+h)-v_n(x) \big)}{|h|^{N+2s}}\phi_\delta(x)\,dh\,dx
\\&&\qquad\qquad-\iint_{\R^{2N}}\frac{\big(u_\varepsilon(x+h)-u_\varepsilon(x)\big)
\big( (v_n)_+(x+h)-(v_n)_+(x) \big)}{|h|^{N+2s}}\phi_\delta(x)\,dh\,dx\Bigg]\\&=&0.
\end{eqnarray*}

Using this information into~\eqref{questainfor7543786qui}, we deduce that
\begin{equation*}
0\ge
\lim_{\delta\to0}\left(\lim_{n\to +\infty}
\int_{\R^N}\phi_\delta|D^s(u_\varepsilon+(v_n)_+)|^2\,dx
-\int_{\R^N}\phi_\delta\,d\nu\right).
\end{equation*}
Accordingly, this and~\eqref{limitephideltaDu+v}
entail that
\begin{equation*}
0\ge
\lim_{\delta\to0}\left(\int_{\R^N}\phi_\delta\,d\mu
-\int_{\R^N}\phi_\delta\,d\nu\right).
\end{equation*}
As a result, recalling~\eqref{CCvmui}
and~\eqref{CCvnui}, we have that~$
0\ge\mu_i-\nu_i$. 
This and~\eqref{CCvSi}
give the desired claim in~\eqref{claimvmunu}.

Similarly, one can show that
\begin{equation}\label{claimvmunu2} {\mbox{if~$\mu_\infty\neq0$, then~$\nu_\infty\ge S^{\frac{N}{2s}}$.}}
\end{equation}
The strategy to prove~\eqref{claimvmunu2} is identical to the one used
to establish~\eqref{claimvmunu}, but using in~\eqref{xherg56cjhiamala43t}
the test function~$w:=(u_\varepsilon+v_n)\phi_R$ with~$
\phi_R:=1-\phi_{R,0}$, for some~$R>0$ to be sent to~$+\infty$.

Now we claim that
\begin{equation}\label{09876543qwertyuio}
{\mbox{$\nu_i=0$, for all~$i\in I$, and~$\nu_\infty=0$.}}
\end{equation}
For this, we exploit~\eqref{CCM} and~\eqref{CCMBIS} to see that
\begin{equation}\label{ijtoy5847854768467476hfuk}\begin{split}
&\frac{s}{N} S^{\frac{N}{2s}}>c=\lim_{n\to+\infty}\left( I_\varepsilon(v_n)-\frac12 \langle
I'_\varepsilon(v_n), v_n\rangle\right)\\&=\lim_{n\to+\infty}\Bigg[
-\frac{\varepsilon}{q+1}\int_{\R^N}h(x)
\Big((u_\varepsilon(x)+(v_n)_+(x))^{q+1}-u_\varepsilon^{q+1}(x)
-(q+1)u_\varepsilon^q(x) (v_n)_+(x)\Big)\,dx
\\&\qquad-\frac{1}{2^*_s}\int_{\R^N}\Big((u_\varepsilon(x)+(v_n)_+(x))^{2^*_s}-u_\varepsilon^{2^*_s}(x)
-2^*_s u_\varepsilon^{2^*_s-1}(x) (v_n)_+(x)\Big)\,dx
\\&\qquad+\frac12\int_{\R^N}\Big[ \varepsilon h(x)\Big( (u_\varepsilon(x)+(v_n)_+(x))^q-u_\varepsilon^q(x)\Big)
+\Big((u_\varepsilon(x)+(v_n)_+(x))^{2^*_s-1}-u_\varepsilon^{2^*_s-1}(x)\Big) \Big]v_n(x)\,dx\Bigg]\\
&\ge\lim_{n\to+\infty}\Bigg[
-\frac{\varepsilon}{q+1}\int_{\R^N}h(x)
\Big((u_\varepsilon(x)+(v_n)_+(x))^{q+1}-u_\varepsilon^{q+1}(x)
-(q+1)u_\varepsilon^q(x) (v_n)_+(x)\Big)\,dx
\\&\qquad-\frac{1}{2^*_s}\int_{\R^N}\Big((u_\varepsilon(x)+(v_n)_+(x))^{2^*_s}-u_\varepsilon^{2^*_s}(x)\Big)\,dx
\\&\qquad+\frac12\int_{\R^N}\Big[ \varepsilon h(x)\Big( (u_\varepsilon(x)+(v_n)_+(x))^q-u_\varepsilon^q(x)\Big)
+\Big((u_\varepsilon(x)+(v_n)_+(x))^{2^*_s-1}-u_\varepsilon^{2^*_s-1}(x)\Big) \Big]v_n(x)\,dx\Bigg]
.
\end{split}\end{equation}

Now we observe that
\begin{equation}\label{eqfinale1}
\lim_{n\to +\infty}\int_{\R^N}h(u_\varepsilon+(v_n)_+)^{q+1}\,dx
=\int_{\R^N}h u_\varepsilon^{q+1}\,dx.
\end{equation}
To show this claim, we take~$R>0$ and notice that,
from the convergence of~$v_n$ and the H\"older inequality,
\[
\begin{aligned}
\lim_{n\to +\infty}&\left|
\int_{\R^N}h(u_\varepsilon+(v_n)_+)^{q+1}\,dx
-\int_{\R^N}h u_\varepsilon^{q+1}\,dx\right| \\
&=\lim_{n\to +\infty}\left[\left|
\int_{B_R}h(u_\varepsilon+(v_n)_+)^{q+1}\,dx
-\int_{\R^N}h u_\varepsilon^{q+1}\,dx\right|
+\left|\int_{\R^N\setminus B_R}h(u_\varepsilon+(v_n)_+)^{q+1}\,dx
\right|\right] \\
&\leq \lim_{n\to +\infty}\Bigg[\left|
\int_{B_R}h(u_\varepsilon+(v_n)_+)^{q+1}\,dx
-\int_{\R^N}h u_\varepsilon^{q+1}\,dx\right|\\
&\qquad\qquad\qquad+
\left(\int_{\R^N\setminus B_R}|h|^\frac{2^*_s}{2^*_s-q-1} \right)^\frac{2^*_s-q-1}{2^*_s}
\left(\int_{\R^N\setminus B_R}|u_\varepsilon+(v_n)_+|^{2^*_s} \right)^\frac{q+1}{2^*_s}\Bigg] \\
&\leq \left|\int_{B_R}hu_\varepsilon^{q+1}\,dx
-\int_{\R^N}h u_\varepsilon^{q+1}\,dx\right|
+C\left(\int_{\R^N\setminus B_R}|h|^\frac{2^*_s}{2^*_s-q-1} \right)^\frac{2^*_s-q-1}{2^*_s}\\
&=\left|\int_{\R^N\setminus B_R}h u_\varepsilon^{q+1}\,dx\right|
+C\left(\int_{\R^N\setminus B_R}|h|^\frac{2^*_s}{2^*_s-q-1} \right)^\frac{2^*_s-q-1}{2^*_s} ,
\end{aligned}
\] which is infinitesimal
as~$R\to +\infty$, thus establishing~\eqref{eqfinale1}.

By plugging~\eqref{eqfinale1}
into~\eqref{ijtoy5847854768467476hfuk}, we get that
\begin{equation*}\begin{split}
&\frac{s}{N} S^{\frac{N}{2s}}>
\lim_{n\to+\infty}\Bigg[
\varepsilon\int_{\R^N}h(x)
u_\varepsilon^q(x) (v_n)_+(x)\,dx
-\frac{1}{2^*_s}\int_{\R^N}\Big((u_\varepsilon(x)+(v_n)_+(x))^{2^*_s}-u_\varepsilon^{2^*_s}(x)\Big)\,dx
\\&+\frac12\int_{\R^N}\Big[ \varepsilon h(x)\Big( (u_\varepsilon(x)+(v_n)_+(x))^q-u_\varepsilon^q(x)\Big)
+\Big((u_\varepsilon(x)+(v_n)_+(x))^{2^*_s-1}-u_\varepsilon^{2^*_s-1}(x)\Big) \Big](v_n)_+(x)\,dx\Bigg]
.
\end{split}\end{equation*}
In light of the weak convergence of~$(v_n)_+$, we also see that
\begin{equation}\label{qwertyuilkjhgfdszxcvbnm098765432}
\lim_{n\to+\infty}\int_{\R^N}h(x)
u_\varepsilon^q(x) (v_n)_+(x)\,dx=0,\end{equation}
and therefore
\begin{equation}\label{ijtoy5847854768467476hfuk2}\begin{split}
&\frac{s}{N} S^{\frac{N}{2s}}>
\lim_{n\to+\infty}\Bigg[
-\frac{1}{2^*_s}\int_{\R^N}\Big((u_\varepsilon(x)+(v_n)_+(x))^{2^*_s}-u_\varepsilon^{2^*_s}(x)\Big)\,dx
\\&\qquad+\frac12\int_{\R^N}\Big[ \varepsilon h(x)
(u_\varepsilon(x)+(v_n)_+(x))^q
+\Big((u_\varepsilon(x)+(v_n)_+(x))^{2^*_s-1}-u_\varepsilon^{2^*_s-1(x)}\Big) \Big](v_n)_+(x)\,dx\Bigg]
.
\end{split}\end{equation}

Furthermore, we claim that
\begin{equation}\label{f4ity48gkgfewugjt098765}
\lim_{n\to+\infty}\int_{\R^N} h(x)
(u_\varepsilon(x)+(v_n)_+(x))^q(v_n)_+(x)\,dx=0.
\end{equation}
Indeed, let~$R>0$ and notice that
the locally strong convergence of~$(v_n)_+$ and the dominated convergence
theorem give that
$$\lim_{n\to+\infty}\int_{B_R} h(x)
\Big((u_\varepsilon(x)+(v_n)_+(x))^q-u_\varepsilon^q(x) \Big)(v_n)_+(x)\,dx
=0.$$
Hence, thanks to~\eqref{qwertyuilkjhgfdszxcvbnm098765432},
and using the H\"older inequality with exponents~$\frac{2^*_s}{2^*_s-q-1}$,
$\frac{2^*_s}{q}$ and~$2^*_s$,
\begin{eqnarray*}&&
\lim_{n\to+\infty}\left|\int_{\R^N} h(x)
(u_\varepsilon(x)+(v_n)_+(x))^q(v_n)_+(x)\,dx\right|
\\&&\qquad=\lim_{n\to+\infty}\left|\int_{\R^N} h(x)\Big(
(u_\varepsilon(x)+(v_n)_+(x))^q-u_\varepsilon^q(x) \Big)(v_n)_+(x)\,dx
\right|\\&&\qquad
=\lim_{n\to+\infty}\left|\int_{\R^N\setminus B_R} h(x)\Big(
(u_\varepsilon(x)+(v_n)_+(x))^q-u_\varepsilon^q(x) \Big)(v_n)_+(x)\,dx
\right|\\&&\qquad\le C
\left(\int_{\R^N\setminus B_R} |h(x)|^{\frac{2^*_s}{2^*_s-q-1}}
\,dx\right)^{\frac{2^*_s-q-1}{2^*_s}}\\&&\qquad\qquad\qquad\cdot
\left(\int_{\R^N}
\big(u_\varepsilon(x)+(v_n)_+(x) \big)^{2^*_s}\,dx\right)^{\frac{q}{2^*_s}}
\left(\int_{\R^N}|(v_n)_+(x)|^{2^*_s}\,dx\right)^{\frac1{2^*_s}}
\\&&\qquad \le C\left(\int_{\R^N\setminus B_R} |h(x)|^{\frac{2^*_s}{2^*_s-q-1}}
\,dx\right)^{\frac{2^*_s-q-1}{2^*_s}},
\end{eqnarray*}
up to renaming~$C>0$. Then, the claim in~\eqref{f4ity48gkgfewugjt098765}
follows by sending~$R\to+\infty$.

Exploiting~\eqref{f4ity48gkgfewugjt098765} into~\eqref{ijtoy5847854768467476hfuk2}, we conclude that
\begin{equation}\label{ijtoy5847854768467476hfuk3}\begin{split}
\frac{s}{N} S^{\frac{N}{2s}}&>
\lim_{n\to+\infty}\Bigg[
-\frac{1}{2^*_s}\int_{\R^N}\Big((u_\varepsilon(x)+(v_n)_+(x))^{2^*_s}-u_\varepsilon^{2^*_s}(x)\Big)\,dx
\\&\qquad+\frac12\int_{\R^N}
\Big((u_\varepsilon(x)+(v_n)_+(x))^{2^*_s-1}-u_\varepsilon^{2^*_s-1}(x)\Big) (v_n)_+(x)\,dx\Bigg]\\&=
\lim_{n\to+\infty}\Bigg[\left(\frac12-\frac{1}{2^*_s}\right)
\int_{\R^N}(u_\varepsilon(x)+(v_n)_+(x))^{2^*_s}\,dx
+\frac{1}{2^*_s}\int_{\R^N}u_\varepsilon^{2^*_s}(x)\,dx\\&\qquad
-\frac12\int_{\R^N}(u_\varepsilon(x)+(v_n)_+(x))^{2^*_s-1}
u_\varepsilon(x)\,dx-\frac12\int_{\R^N}u_\varepsilon^{2^*_s-1}(x) (v_n)_+(x)\,dx\Bigg]
.\end{split}\end{equation}
We now point out that
\begin{equation}\label{speriamobenechicelofafareeg}
\lim_{n\to+\infty}\int_{\R^N}(u_\varepsilon(x)+(v_n)_+(x))^{2^*_s-1}
u_\varepsilon(x)\,dx=
\int_{\R^N}u_\varepsilon^{2^*_s}(x)\,dx.
\end{equation}
To check this, we pick~$R>0$ and use the locally strong convergence of~$(v_n)_+$ and the dominated convergence
theorem to see that
$$ \lim_{n\to+\infty}\int_{B_R}(u_\varepsilon(x)+(v_n)_+(x))^{2^*_s-1}
u_\varepsilon(x)\,dx=\int_{B_R}u_\varepsilon^{2^*_s}(x)\,dx$$and consequently
\begin{eqnarray*}&&
\lim_{n\to+\infty}\left|\int_{\R^N}(u_\varepsilon(x)+(v_n)_+(x))^{2^*_s-1}
u_\varepsilon(x)\,dx-\int_{\R^N}u_\varepsilon^{2^*_s}(x)\,dx\right|
\\&&\qquad=\lim_{n\to+\infty}\left|\int_{\R^N\setminus B_R}(u_\varepsilon(x)+(v_n)_+(x))^{2^*_s-1}
u_\varepsilon(x)\,dx-\int_{\R^N\setminus B_R}u_\varepsilon^{2^*_s}(x)\,dx\right|\\
&&\qquad\le
C\lim_{n\to+\infty}\left(
\left(\int_{\R^N} |(v_n)_+(x)|^{2^*_s}\,dx\right)^{\frac{2^*_s-1}{2^*_s}}
\left(\int_{\R^N\setminus B_R}u_\varepsilon^{2^*_s}(x)\,dx\right)^{\frac1{2^*_s}}
+
\int_{\R^N\setminus B_R}u_\varepsilon^{2^*_s}(x)\,dx
\right)\\&&\qquad\le C\left(
\left(\int_{\R^N\setminus B_R}u_\varepsilon^{2^*_s}(x)\,dx\right)^{\frac1{2^*_s}}
+
\int_{\R^N\setminus B_R}u_\varepsilon^{2^*_s}(x)\,dx
\right),
\end{eqnarray*}
up to renaming~$C>0$, and thus,
by sending~$R\to+\infty$,
we establish the desired claim in~\eqref{speriamobenechicelofafareeg}.

Moreover, the weak convergence of~$(v_n)_+$ entails that
\begin{equation}\label{4tbhdsluypuitsgetrire}
\lim_{n\to+\infty}\int_{\R^N}u_\varepsilon^{2^*_s-1}(x) (v_n)_+(x)\,dx=0.
\end{equation}
Plugging this and~\eqref{speriamobenechicelofafareeg} into~\eqref{ijtoy5847854768467476hfuk3}, we obtain that
\begin{equation*}\begin{split}
\frac{s}{N} S^{\frac{N}{2s}}&>\left(\frac12-\frac{1}{2^*_s}\right)
\lim_{n\to+\infty}\Bigg[
\int_{\R^N}(u_\varepsilon(x)+(v_n)_+(x))^{2^*_s}\,dx
-\int_{\R^N}u_\varepsilon^{2^*_s}(x)\,dx\Bigg]
.\end{split}\end{equation*}
{F}rom this, \eqref{CCvnui} and~\eqref{CCvnuinfinito} we conclude that
\begin{equation*}
\frac{s}{N} S^{\frac{N}{2s}}>\left(\frac12-\frac{1}{2^*_s}\right)\left(
\sum_{i\in I}\nu_i+\nu_\infty\right)=\frac{s}N
\left(\sum_{{i\in I}\atop{\nu_i\ne0}}\nu_i+\nu_\infty\right)
.\end{equation*}
This, \eqref{CCvSi}, \eqref{CCvSinfinito},
\eqref{claimvmunu} and~\eqref{claimvmunu2} entail that~$\nu_i=0$
for all~$i\in I$ and that~$\nu_\infty=0$. The proof of~\eqref{09876543qwertyuio} is thereby complete.

{F}rom~\eqref{CCvnui}, \eqref{CCvSinfinito}
and~\eqref{09876543qwertyuio}, we gather that
\begin{equation}\label{dijweoyfgeu4ty8684tughnnvbcjhsvdjwe}
\lim_{n\to+\infty}
\int_{\R^N}\big(u_\varepsilon(x)+(v_n)_+(x)\big)^{2^*_s}\,dx
=\int_{\R^N}u_\varepsilon^{2^*_s}(x)\,dx.\end{equation}

Now, in light of~\eqref{CCMBIS},
\begin{eqnarray*}
0&=&\lim_{n\to+\infty}\langle I'_\varepsilon(v_n),v_n\rangle\\
&=& \lim_{n\to+\infty}\Bigg[\int_{\R^{N}}|D^s v_n(x)|^2\,dx
\\&&-\int_{\R^N}\Big[ \varepsilon h(x)\Big( (u_\varepsilon(x)+(v_n)_+(x))^q-u_\varepsilon^q\Big)
+\Big((u_\varepsilon(x)+(v_n)_+(x))^{2^*_s-1}-u_\varepsilon^{2^*_s-1}\Big) \Big](v_n)_+(x)\,dx\Bigg].
\end{eqnarray*}
Accordingly, exploiting~\eqref{qwertyuilkjhgfdszxcvbnm098765432},
\eqref{f4ity48gkgfewugjt098765} and~\eqref{4tbhdsluypuitsgetrire},
\begin{eqnarray*}
0&=& \lim_{n\to+\infty}\int_{\R^{N}}|D^s v_n(x)|^2\,dx
-\int_{\R^N}(u_\varepsilon(x)+(v_n)_+(x))^{2^*_s-1}(v_n)_+(x)\,dx\\
&=&\lim_{n\to+\infty}\Bigg[\int_{\R^{N}}|D^s v_n(x)|^2\,dx
-\int_{\R^N}(u_\varepsilon(x)+(v_n)_+(x))^{2^*_s}\,dx
\\&&\qquad\qquad+\int_{\R^N}(u_\varepsilon(x)+(v_n)_+(x))^{2^*_s-1}u_\varepsilon(x)\,dx\Bigg]
.\end{eqnarray*}
We can now use~\eqref{speriamobenechicelofafareeg}
and~\eqref{dijweoyfgeu4ty8684tughnnvbcjhsvdjwe} to conclude that
\begin{eqnarray*}
0&=&\lim_{n\to+\infty}\Bigg[\int_{\R^{N}}|D^s v_n(x)|^2\,dx
-\int_{\R^N}u_\varepsilon^{2^*_s}(x)\,dx
+\int_{\R^N}u_\varepsilon^{2^*_s}\,dx\Bigg] \\&=&
\lim_{n\to+\infty}\int_{\R^{N}}|D^s v_n(x)|^2\,dx,
\end{eqnarray*}
which completes the proof of Proposition~\ref{PSIepsilon}. 
\end{proof}

Now, in order to prove the existence of a second solution, we consider 
solutions~$u\in\mathcal{D}^{s,2}(\R^N)$ to the equation
\[
(-\Delta)^s u=u^{2^*_s-1}.
\]
Notice that this equation has a manifold~$Z$ of positive solutions
\[
Z:=\left\{z_{\mu,\xi}(x)=\mu^{-\frac{N-2s}{2}}z\left(\frac{x-\xi}{\mu}\right){\mbox{ with }}\, \mu>0, \xi\in \R^N \right\},
\]
where
\begin{equation}\label{definizionez}
z(x):=\frac{C_{N,s}}{(1+|x|^2)^\frac{N-2s}{2}}
\end{equation}
for a suitable~$C_{N,s}>0$.

Furthermore, we set
\begin{equation}\label{degtfte093654BIS}\begin{split}
\widetilde{G}(x,v)&:=\int_0^v \widetilde{g}(x,t)dt
\\
{\mbox{and }} \quad\widetilde{g}(x,t)&:=
\begin{cases}
(u_\varepsilon +t)^q-u_\varepsilon^q  & \mbox{if } t\geq 0,
\\
0   & \mbox{if } t< 0.
\end{cases}
\end{split}\end{equation}
We observe that 
\begin{equation}\label{eq6.5.2}
\widetilde{G}(x,0)=0
\end{equation}
and also, since~$\widetilde{g}(x,t)\geq 0$ for any~$t\in \R$, that
\begin{equation}\label{eq6.5.3}
\widetilde{G}(x,v)\geq 0 \mbox{ for any } v\geq 0.
\end{equation}

Moreover, we write the ball~$B$ given by~$(h_1)$
as~$B_{\mu_0}(\xi)$ for some~$\xi \in \R^N$ and~$\mu_0>0$.
In this way,
\begin{equation}\label{djiweoty8e4otyvbc5674t657485748748hgfjdfghddh}
{\mbox{$h>0$ in~$B_{\mu_0}(\xi)$.}}
\end{equation}
We point out that~$\xi$ and~$\mu_0$ are now fixed once and for all.

We pick a cut-off function~$\phi\in C_c^\infty(B_{\mu_0}(\xi),[0,1])$ with
\begin{equation}\label{eq6.5.4}
\phi(x)=1 \mbox{ for any } x\in B_{\mu_0/2}(\xi).
\end{equation}
Also, for any~$\mu>0$, recalling~\eqref{definizionez}, we set
\begin{equation}\label{eq6.5.5}
z_{\mu,\xi}(x):=\mu^{-\frac{N-2s}{2}}z\left(\frac{x-\xi}{\mu}\right).
\end{equation}
By scaling,
\begin{equation}\label{eq6.5.7}
[z_{\mu,\xi}]_{s}^2=\|z_{\mu,\xi}\|_{L^{2^*_s}(\R^N)}^{2^*_s}=S^\frac{N}{2s}.
\end{equation}
In addition, by Proposition~21 in~\cite{MR3271254},
\begin{equation}\label{eq6.5.8}
[\phi z_{\mu,\xi}]_{s}^2\leq S^\frac{N}{2s}+C\mu^{N-2s}
\end{equation}
and,
by Lemma~6.5.1 in~\cite{MR3617721},
\begin{equation}\label{lemma6.5.1}
\int_{\R^N} \big(1-\phi^{2^*_s}(x)\big) z_{\mu,\xi}^{2^*_s}(x)\,dx\leq C\mu^N,
\end{equation}
for some~$C>0$.

The next result states that we can concentrate the mass
near the set where~$h$ is positive to obtain a positive integral.

\begin{lem}\label{eq6.5.9}
For any~$\mu>0$ and~$t\geq 0$, we have that
\begin{equation}
\int_{\R^N}h(x)\,\widetilde{G}(x,t\phi(x)z_{\mu,\xi}(x))\,dx\geq 0.
\end{equation}
\end{lem}

\begin{proof}
Since~$\phi(x)=0$ if~$x\in \R^N\setminus B_{\mu_0}(\xi)$, using~\eqref{eq6.5.2} we have that~$
\widetilde{G}(x,t\phi(x)z_{\mu,\xi}(x))=0$
for any~$x\in \R^N\setminus B_{\mu_0}(\xi)$. 

As a result,
\[
\int_{\R^N}h(x)\,\widetilde{G}(x,t\phi(x)z_{\mu,\xi}(x))\,dx
=\int_{B_{\mu_0}(\xi)}h(x)\,\widetilde{G}(x,t\phi(x)z_{\mu,\xi}(x))\,dx.
\]
{F}rom~$(h_1)$ and~\eqref{eq6.5.3} we obtain the desired result.
\end{proof}

Now we show that the functional~$I_\varepsilon$ satisfies the 
geometry of the Mountain Pass Theorem. To this end, we first point out that Proposition~\ref{propvminimolocale} gives that~$0$ is a local minimum for~$I_\varepsilon$. 

Now, we show that the path induced by the function~$\phi z_{\mu,\xi}$ attains negative values. In order to do this,
we introduce the auxiliary functional
\begin{equation}\label{6.5.10}\begin{split}
&I_\varepsilon^\star(v):=\frac{1}{2}[v]_{s}^2
-\int_{\R^N}G^\star (x,v)\,dx,
\\{\mbox{where }} \quad &
G^\star(x,v):=\int_0^v g^\star(x,t)\,dt\\ {\mbox{and }}\quad&
g^\star(x,v):=
\begin{cases}
(u_\varepsilon +t)^{2^*_s-1}-u_\varepsilon^{2^*_s-1}  &\mbox{if }
t\geq 0,  \\
0    &\mbox{if } t< 0.
\end{cases}
\end{split}\end{equation}
We have the following result:

\begin{lem}\label{lemma6.5.3BIS}
There exists~$\mu_1\in (0,\mu_0)$ such that
\[
\lim_{t\to +\infty} \sup_{\mu\in (0,\mu_1)}
I^\star_\varepsilon(t\phi z_{\mu,\xi})=-\infty.
\]
In particular, there exists~$T_1>0$ such that
\begin{equation}\label{eq6.5.12BIS}
\sup_{\mu\in (0,\mu_1)}I^\star_\varepsilon(t\phi z_{\mu,\xi})\leq 0
\end{equation}
for any~$t\geq T_1$.
\end{lem}

\begin{proof}
We observe that, if~$v\geq 0$,
\begin{equation}\label{eq6.5.13}
G^\star (x,v)=\int_0^v \Big( (u_\varepsilon+t)^{2^*_s-1}-u_\varepsilon^{2^*_s-1}\Big)\,dt
=\frac{1}{2^*_s}\Big((u_\varepsilon+v)^{2^*_s}-u_\varepsilon^{2^*_s}\Big)
-u_\varepsilon^{2^*_s-1}v.
\end{equation}
Moreover, by the Young inequality,
\[
u_\varepsilon^{2^*_s-1}(x)\phi(x) z_{\mu,\xi}(x)
\leq u_\varepsilon^{2^*_s}(x)+z_{\mu,\xi}^{2^*_s}(x).
\]
{F}rom this and~\eqref{eq6.5.13} we obtain that
\[
\begin{aligned}
G^\star(x,t\phi z_{\mu,\xi})
&=\frac{1}{2^*_s}\Big((u_\varepsilon+t\phi z_{\mu,\xi})^{2^*_s}-u_\varepsilon^{2^*_s}\Big)
-u_\varepsilon^{2^*_s-1}t\phi z_{\mu,\xi}  \\
&\geq \frac{1}{2^*_s}\Big((t\phi z_{\mu,\xi})^{2^*_s}-u_\varepsilon^{2^*_s}\Big)
-tu_\varepsilon^{2^*_s}-tz_{\mu,\xi}^{2^*_s}.
\end{aligned}
\]
Thus, integrating over~$\R^N$ and recalling~$(h_0)$ and~\eqref{eq6.5.7}, we gather that
\begin{equation}\label{eq6.5.14}
\int_{\R^N}G^\star(x,t\phi z_{\mu,\xi})\,dx
\geq \frac{t^{2^*_s}}{2^*_s}\int_{\R^N}\phi^{2^*_s} z_{\mu,\xi}^{2^*_s}\,dx-C -Ct
\end{equation}
for some~$C>0$.

Furthermore, from~\eqref{eq6.5.7} and~\eqref{lemma6.5.1} we deduce that there exists~$\mu_1\in(0,1)$
such that if~$\mu\in (0,\mu_1)$ then
\[
\int_{\R^N}\phi^{2^*_s} z_{\mu,\xi}^{2^*_s}\,dx
\geq \int_{\R^N} z_{\mu,\xi}^{2^*_s}\,dx
-C\mu^N=S^\frac{N}{2s}-C\mu^N
\ge
\frac{1}{2}S^\frac{N}{2s}.
\]
Inserting this in~\eqref{eq6.5.14} we obtain that, if~$\mu \in (0,\mu_1)$, 
\[
\int_{\R^N}G^\star(x,t\phi z_{\mu,\xi})\,dx
\geq \frac{t^{2^*_s}}{2\cdot2^*_s}S^\frac{N}{2s}-C -Ct.
\]
This and~\eqref{6.5.10} give that
\[
I_\varepsilon^\star(t\phi z_{\mu,\xi})
\leq \frac{t^2}{2}[\phi z_{\mu,\xi}]_{s}+C(1+t)
-\frac{t^{2^*_s}}{2\cdot2^*_s}S^\frac{N}{2s}.
\]
Hence, recalling~\eqref{eq6.5.8}, we find that
\[
I_\varepsilon^\star(t\phi z_{\mu,\xi})
\leq C(1+t+t^2)-\frac{t^{2^*_s}}{2\cdot2^*_s}S^\frac{N}{2s},
\]
up to renaming constants, for any~$\mu\in (0,\mu_1)$.
Since~$2^*_s>2$, we obtain the desired result.
\end{proof}

As a consequence of Lemma~\ref{lemma6.5.3BIS} we have the following:

\begin{cor}\label{lemma6.5.3}
There exists~$\mu_1\in (0,\mu_0)$ such that
\[
\lim_{t\to +\infty} \sup_{\mu\in (0,\mu_1)}
I_\varepsilon(t\phi z_{\mu,\xi})=-\infty.
\]
In particular, there exists~$T_1>0$ such that
\begin{equation}\label{eq6.5.12}
\sup_{\mu\in (0,\mu_1)}I_\varepsilon(t\phi z_{\mu,\xi})\leq 0
\end{equation}
for any~$t\geq T_1$.
\end{cor}

\begin{proof}
Recalling the definition of~$G$ and~$\widetilde{G}$ in~\eqref{degtfte093654}
and~\eqref{degtfte093654BIS} respectively, we can write~$G=G^\star+\varepsilon h \widetilde{G}$. Thus, from Lemma~\ref{eq6.5.9}
we see that
\begin{equation}\label{feiwhwbnrgjkthioi8eu56895}
I_\varepsilon(t\phi z_{\mu,\xi})=
I_\varepsilon^\star(t\phi z_{\mu,\xi})
-\varepsilon \int_{\R^N}h(x)\widetilde{G}(x,t\phi(x)z_{\mu,\xi}(x))\,dx
\leq I_\varepsilon^\star(t\phi z_{\mu,\xi}).
\end{equation}
{F}rom this and Lemma~\ref{lemma6.5.3BIS} we obtain the desired result.
\end{proof}

In addition, we have the following result for
the functional~$I_\varepsilon^\star$.

\begin{lem}\label{lemma6.5.7}
There exists~$\mu_\star\in (0,\mu_0)$ such that if~$\mu\in (0,\mu_\star)$ then
\begin{equation}\label{eq6.5.17}
\sup_{t\geq 0}I_\varepsilon^\star(t\phi z_{\mu,\xi})
<\frac{s}{N}S^\frac{N}{2s}.
\end{equation} 
\end{lem}

\begin{proof}
Let~$\mu_1>0$ and~$T_1>0$ as in Lemma~\ref{lemma6.5.3BIS}.
We take~$\mu_\star\in(0, \mu_1]$. Without loss of generality, we can suppose
that~$\mu_\star\le\frac{\mu_0^2}{16}$ whenever needed.

In this way, in light of~\eqref{eq6.5.12BIS} we see that
\begin{equation}\label{eq6.5.18}
\sup_{t\geq T_1}\sup_{\mu\in(0,\mu_\star)}
I_\varepsilon^\star(t\phi z_{\mu,\xi})
\leq 
\sup_{t\geq T_1}\sup_{\mu\in(0,\mu_1)}
I_\varepsilon^\star(t\phi z_{\mu,\xi})
\leq 0<\frac{s}{N}S^\frac{N}{2s}.
\end{equation}
Hence, the claim in~\eqref{eq6.5.17} will be completed if we prove that,
for any~$\mu\in (0,\mu_\star)$,
\begin{equation}\label{eq6.5.19}
\sup_{t\in [0,T_1]}I_\varepsilon^\star(t\phi z_{\mu,\xi})
<\frac{s}{N}S^\frac{N}{2s}.
\end{equation}
In order to do so, we set
\begin{equation}\label{eq6.5.20}
m:=
\begin{cases}
2      &\mbox{if } N>4s,
\\
2^*_s-1  &\mbox{if } N\in (2s,4s],
\end{cases}
\end{equation}
and
\begin{equation}\label{eq6.5.21}
\Omega:=
\begin{cases}
B_{2\sqrt{\mu}}(\xi)\setminus B_{\sqrt{\mu}}(\xi) &\mbox{if } N>4s,
\\
\R^N  &\mbox{if } N\in (2s,4s].
\end{cases}
\end{equation}
We notice that, if~$N\in (2s,4s]$, then~$m-2=\frac{6s-N}{N-2s}>0$, and so
\begin{equation}\label{eq6.5.22}
m-2\geq 0 \mbox{ for every } N>2s.
\end{equation}

We claim that, for any~$t\in [0,T_1]$, any~$\mu\in (0,\mu_\star)$
and any~$x\in \Omega$,
\begin{equation}\label{eq6.5.23}
G^\star(x,t\phi(x) z_{\mu,\xi}(x))
\geq \frac{t^{2^*_s}\phi^{2^*_s}(x) z_{\mu,\xi}^{2^*_s}(x)}{2^*_s} 
+\frac{cu_\varepsilon^{2^*_s-m}(x)t^m\phi^m(x) z_{\mu,\xi}^m(x)}{m}
,\end{equation}
for some~$c>0$.

To prove this claim, 
we treat separately the cases~$N>4s$ and~$N\in (2s,4s]$. If~$N>4s$, 
we take~$a:=u_\varepsilon(x)$ and
\[
b\in\left[0, t\inf_{x\in \Omega}\phi(x) z_{\mu,\xi}(x)\right].
\]
We point out that
$$ a\ge \inf_{\Omega} u_\varepsilon \ge \inf_{B_{\mu_0/2}(\xi)}u_\varepsilon=:j_\varepsilon$$
and we stress that~$j_\varepsilon>0$, thanks to~\eqref{djiweoty8e4otyvbc5674t657485748748hgfjdfghddh} and
Proposition~\ref{prop:primasol}.

Moreover, recalling the definition of~$z$ in~\eqref{definizionez},
\[
b\leq t\phi(x)z_{\mu,\xi}(x)\leq tz_{\mu,\xi}(x)=
t\mu^{-\frac{N-2s}{2}}z\left(\frac{x-\xi}{\mu}\right)
=\frac{C_{N,s}t\mu^{-\frac{N-2s}{2}}}{\left(1+\left|\frac{x-\xi}{\mu}\right|^2\right)^\frac{N-2s}{2}}
=\frac{C_{N,s}t\mu^\frac{N-2s}{2}}{(\mu^2+|x-\xi|^2)^\frac{N-2s}{2}}.
\]
Since~$x\in B_{2\sqrt{\mu}}(\xi)\setminus B_{\sqrt{\mu}}(\xi)$,
we have that~$|x-\xi|\geq \sqrt{\mu}$, and so
\[
b\leq \frac{C_{N,s}t\mu^\frac{N-2s}{2}}{(\mu^2+\mu)^\frac{N-2s}{2}}
\leq \frac{C_{N,s}t\mu^\frac{N-2s}{2}}{\mu^\frac{N-2s}{2}}=C_{N,s}t
\leq C_{N,s}T_1.
\]
{F}rom this, we obtain that~$\frac{b}{a}\leq k$
with~$k:=\frac{C_{N,s}T_1}{j_\varepsilon}$, hence we can apply Lemma~\ref{lemma6.5.6} to
obtain that
\[
\begin{aligned}
G^\star(x,t\phi (x)z_{\mu,\xi}(x))&=
\int_0^{t\phi(x) z_{\mu,\xi}(x)}\Big((u_\varepsilon(x)+b)^{2^*_s-1}
-u_\varepsilon^{2^*_s-1}(x)\Big)\,db \\
&=\int_0^{t\phi(x) z_{\mu,\xi}(x)}\Big( (a+b)^{2^*_s-1}
-a^{2^*_s-1}\Big)\,db \\
&\geq \int_0^{t\phi(x) z_{\mu,\xi}(x)}\Big(b^{2^*_s-1}+c_{2^*_s-1,k}a^{2^*_s-2}b\Big)\,db \\
&=\frac{(t\phi(x) z_{\mu,\xi}(x))^{2^*_s}}{2^*_s}
+c_{2^*_s-1,k}u_\varepsilon^{2^*_s-2}(x)\frac{(t\phi(x) z_{\mu,\xi}(x))^2}{2}.
\end{aligned}
\]
This and~\eqref{eq6.5.20} yield~\eqref{eq6.5.23} when~$N>4s$.

To prove~\eqref{eq6.5.23} when~$N\in (2s,4s]$, we first observe 
that 
\[
2^*_s-1=\frac{N+2s}{N-2s}\geq 2.
\]
Hence, we choose $$a\in\left[0, \inf_{x\in\R^N}t\phi(x) z_{\mu,\xi}(x)\right]$$ and~$b:=u_\varepsilon(x)$ and we use Lemma~\ref{lemma6.5.5} to see that
\[
\begin{aligned}
G^\star(x,t\phi (x)z_{\mu,\xi}(x))&=
\int_0^{t\phi (x)z_{\mu,\xi}(x)}\Big( (u_\varepsilon(x)+a)^{2^*_s-1}
-u_\varepsilon^{2^*_s-1}(x)\Big)\,da \\
&=\int_0^{t\phi(x) z_{\mu,\xi}(x)}\Big((a+b)^{2^*_s-1}
-b^{2^*_s-1}\Big)\,da  \\ 
&\geq \int_0^{t\phi (x)z_{\mu,\xi}(x)}\Big(a^{2^*_s-1}
-c_{2^*_s-1} a^{2^*_s-2}b\Big)\,da \\
&=\frac{(t\phi (x)z_{\mu,\xi}(x))^{2^*_s}}{2^*_s}
+c_{2^*_s-1}u_\varepsilon(x) \frac{(t\phi (x)z_{\mu,\xi}(x))^{2^*_s-1}}{2^*_s-1}.
\end{aligned} 
\]
This and~\eqref{eq6.5.20} give~\eqref{eq6.5.23} when~$N\in (2s,4s]$, thus completing the proof of~\eqref{eq6.5.23}.

Now we claim that, for any~$t\in [0,T_1]$, 
$\mu\in (0,\mu_\star)$ and~$x\in \R^N$,
\begin{equation}\label{eq6.5.25}
G^\star(x,t\phi (x)z_{\mu,\xi}(x))
\geq \frac{t^{2^*_s}\phi^{2^*_s}(x) z_{\mu,\xi}^{2^*_s}(x)}{2^*_s}.
\end{equation}
We remark that~\eqref{eq6.5.23} is a stronger inequality than~\eqref{eq6.5.25},
but it only holds in~$\Omega$,
while~\eqref{eq6.5.25} holds in the whole of~$\R^N$, and this is 
an advantage when~$N>4s$ (recall the definition of~$\Omega$ in~\eqref{eq6.5.21}).

In order to prove~\eqref{eq6.5.25}, we use Lemma~\ref{lemma6.5.4}
with~$a:=u_\varepsilon(x)$ and~$b\in[ 0,+\infty)$ to see that
\[
\begin{aligned}
G^\star(x,t\phi(x) z_{\mu,\xi}(x))&=
\int_0^{t\phi(x) z_{\mu,\xi}(x)}\Big( (u_\varepsilon(x)+b)^{2^*_s-1}
-u_\varepsilon^{2^*_s-1}(x)\Big)\,db \\
&=\int_0^{t\phi(x) z_{\mu,\xi}(x)}\Big((a+b)^{2^*_s-1}
-a^{2^*_s-1}\Big)\,db \\
&\geq \int_0^{t\phi (x)z_{\mu,\xi}(x)}b^{2^*_s-1}\,db\\& 
=\frac{(t\phi (x)z_{\mu,\xi}(x))^{2^*_s}}{2^*_s},
\end{aligned}
\]
which establishes~\eqref{eq6.5.25}.

Combining~\eqref{eq6.5.23} and~\eqref{eq6.5.25}, we obtain that
\begin{equation}\label{eq6.5.26}
\begin{aligned}&
\int_{\R^N}G^\star(x,t\phi(x) z_{\mu,\xi}(x))\,dx\\
&=\int_{\R^N \setminus \Omega}G^\star(x,t\phi (x)z_{\mu,\xi}(x))\,dx
+\int_{\Omega}G^\star(x,t\phi(x) z_{\mu,\xi}(x))\,dx \\
&\geq \int_{\R^N \setminus \Omega}
\frac{t^{2^*_s}\phi^{2^*_s}(x) z_{\mu,\xi}^{2^*_s}(x)}{2^*_s}\,dx
+\int_{\Omega}\left(
\frac{t^{2^*_s}\phi^{2^*_s} (x)z_{\mu,\xi}^{2^*_s}(x)}{2^*_s} 
+\frac{cu_\varepsilon^{2^*_s-m}(x)t^m\phi^m(x) z_{\mu,\xi}^m(x)}{m}
\right)\,dx \\
&=\frac{t^{2^*_s}}{2^*_s}
\int_{\R^N}\phi^{2^*_s}(x) z_{\mu,\xi}^{2^*_s}(x)\,dx
+\frac{ct^m}{m}
\int_{\Omega}u_\varepsilon^{2^*_s-m}(x)\phi^m (x)z_{\mu,\xi}^m(x)\,dx.
\end{aligned}
\end{equation}

Now, we show that
\begin{equation}\label{eq6.5.27}
\int_{\Omega}u_\varepsilon^{2^*_s-m}(x)\phi^m (x)z_{\mu,\xi}^m(x)\,dx
\geq c'\mu^\beta,
\end{equation}
for some~$c'>0$, where
\begin{equation}\label{eq6.5.28}
\beta:=
\begin{cases}
\dfrac{N}{2}   & \mbox{if } N>4s,
\\ \\
\dfrac{N-2s}{2} & \mbox{if } N\in (2s,4s].
\end{cases}
\end{equation}
To prove this, when~$N>4s$ we observe that, for small~$\mu$, we have 
that~$B_{2\sqrt{\mu}}(\xi)\subset B_{\mu_0/2}(\xi)$, and in this set~$\phi(x)=1$ thanks to~\eqref{eq6.5.4}. 

Hence, recalling~\eqref{eq6.5.20} and~\eqref{eq6.5.21} we have that
\[
\begin{aligned}
\int_{\Omega}u_\varepsilon^{2^*_s-m}(x)\phi^m(x) z_{\mu,\xi}^m(x)\,dx
&=\int_{B_{2\sqrt{\mu}}(\xi)\setminus B_{\sqrt{\mu}}(\xi)}
u_\varepsilon^{2^*_s-2}(x) z_{\mu,\xi}^2(x)\,dx \\
&\geq \inf_{B_{\mu_0/2}(\xi)}u_\varepsilon^{2^*_s-2}
\int_{B_{2\sqrt{\mu}}(\xi)\setminus B_{\sqrt{\mu}}(\xi)}
z_{\mu,\xi}^2(x)\,dx \\
&= \mu^{-(N-2s)}\inf_{B_{\mu_0/2}(\xi)}u_\varepsilon^{2^*_s-2}
\int_{B_{2\sqrt{\mu}}(\xi)\setminus B_{\sqrt{\mu}}(\xi)}
z^2\left(\frac{x-\xi}{\mu}\right)\,dx \\
&=\mu^{2s}\inf_{B_{\mu_0/2}(\xi)}u_\varepsilon^{2^*_s-2} 
\int_{B_{{2}/{\sqrt{\mu}}}\setminus B_{{1}/{\sqrt{\mu}}}} z^2(y)\,dy.
\end{aligned}
\]
Thus, recalling~\eqref{definizionez} and taking~$\mu$ small enough,
we have
\[
\begin{aligned}
&\int_{\Omega}u_\varepsilon^{2^*_s-m}(x)\phi^m (x)z_{\mu,\xi}^m(x)\,dx
\geq c_1 \mu^{2s} \int_{1/\sqrt{\mu}}^{2/\sqrt{\mu}}
\frac{\rho^{N-1}}{(1+\rho^2)^{N-2s}}\,d\rho \\
&\qquad \geq c_1 \mu^{2s} \int_{1/\sqrt{\mu}}^{2/\sqrt{\mu}}
\frac{\rho^{N-1}}{(2\rho^2)^{N-2s}}\,d\rho
=c_2 \mu^\frac{N}{2},
\end{aligned}
\]
for some~$c_1$, $c_2>0$, and this proves~\eqref{eq6.5.27} when~$N>4s$.

Now we prove~\eqref{eq6.5.27} when~$N\in (2s,4s]$. For this, we recall~\eqref{eq6.5.20} and~\eqref{eq6.5.21} and observe that, for~$\mu$
sufficiently small,
\[
\begin{aligned}
\int_{\Omega}u_\varepsilon^{2^*_s-m}(x)\phi^m(x) z_{\mu,\xi}^m(x)\,dx
&=\int_{\R^N}u_\varepsilon(x)\phi^{2^*_s-1}(x) z_{\mu,\xi}^{2^*_s-1}(x)\,dx \\
&\geq \mu^{-\frac{N+2s}{2}}\int_{B_{2\sqrt{\mu}}(\xi)}
u_\varepsilon(x) z^{2^*_s-1}\left(\frac{x-\xi}{\mu}\right)\,dx \\
&\geq \mu^{-\frac{N+2s}{2}} \inf_{B_{\mu_0/2}(\xi)}u_\varepsilon
\int_{B_{2\sqrt{\mu}}(\xi)}
z^{2^*_s-1}\left(\frac{x-\xi}{\mu}\right)\,dx \\
&=\mu^{\frac{N-2s}{2}}\inf_{B_{\mu_0/2}(\xi)}u_\varepsilon
\int_{B_{{2}/{\sqrt{\mu}}}}z^{2^*_s-1}(y)\,dy \\
&\geq \mu^{\frac{N-2s}{2}}\inf_{B_{\mu_0/2}(\xi)}u_\varepsilon
\int_{B_1}z^{2^*_s-1}(y)\,dy \\
&\geq c'\mu^{\frac{N-2s}{2}},
\end{aligned}
\]
for some~$c'>0$, which concludes the proof in the case~$N\in (2s,4s]$.
Thus, the claim in~\eqref{eq6.5.27} is established.

Now, using~\eqref{eq6.5.27} in~\eqref{eq6.5.26}, we obtain that
\begin{equation*}
\int_{\R^N}G^\star(x,t\phi (x)z_{\mu,\xi}(x))\,dx
\geq \frac{t^{2^*_s}}{2^*_s}
\int_{\R^N}\phi^{2^*_s}(x) z_{\mu,\xi}^{2^*_s}(x)\,dx
+\frac{c\,t^m\mu^\beta}{m},
\end{equation*}
up to renaming~$c>0$. 

{F}rom this and~\eqref{lemma6.5.1}, we infer that
\[
\int_{\R^N}G^\star(x,t\phi z_{\mu,\xi})\,dx
\geq \frac{t^{2^*_s}}{2^*_s}
\int_{\R^N} z_{\mu,\xi}^{2^*_s}\,dx
+\frac{c\,t^m\mu^\beta}{m}
-\frac{Ct^{2^*_s}\mu^N}{2^*_s}.
\]
This and~\eqref{eq6.5.7} give that
\[
\int_{\R^N}G^\star(x,t\phi(x) z_{\mu,\xi}(x))\,dx
\geq \frac{t^{2^*_s}}{2^*_s}S^\frac{N}{2s}
+\frac{c\,t^m\mu^\beta}{m}
-\frac{Ct^{2^*_s}\mu^N}{2^*_s}.
\]
As a consequence, recalling~\eqref{eq6.5.8}, up to renaming constants,
\[
\begin{aligned}
I_\varepsilon^\star (t\phi z_{\mu,\xi})
&\leq \frac{t^2}{2}[\phi z_{\mu,\xi}]_{s}
-\frac{t^{2^*_s}}{2^*_s}S^\frac{N}{2s}
-\frac{ct^m\mu^\beta}{m}
+\frac{Ct^{2^*_s}\mu^N}{2^*_s} \\
&\leq \frac{t^2}{2}S^\frac{N}{2s}
-\frac{t^{2^*_s}}{2^*_s}S^\frac{N}{2s}
-\frac{c\,t^m\mu^\beta}{m}
+\frac{Ct^{2^*_s}\mu^N}{2^*_s}
+\frac{Ct^2 \mu^{N-2s}}{2},
\end{aligned}
\]
and so
\begin{equation}\label{eq6.5.30}
I_\varepsilon^\star(t\phi z_{\mu,\xi})
\leq S^\frac{N}{2s} \Psi(t),
\end{equation}
with
\[
\Psi(t):=
\frac{t^2}{2}-\frac{t^{2^*_s}}{2^*_s}
-\frac{c\,t^m\mu^\beta}{m}+\frac{Ct^{2^*_s}\mu^N}{2^*_s}
+\frac{Ct^2 \mu^{N-2s}}{2}
\]
for some~$c$, $C>0$.

Now we claim that, if~$\mu$ is sufficiently small,
\begin{equation}\label{eq6.5.31}
\sup_{t\geq 0}\Psi(t)<\frac{s}{N}.
\end{equation}
To show this, we observe that~$\Psi(0)=0$ and, if~$\mu$ is sufficiently small,
\[
\lim_{t\to +\infty}\Psi(t)=-\infty,
\]
since~$2^*_s>\max\{2,m\}$, thanks to~\eqref{eq6.5.20}.

On this account, we conclude that~$\Psi$ attains its maximum at some point~$T\in [0,+\infty)$. Clearly, it~$T=0$, then~$\Psi(T)=0$ and~\eqref{eq6.5.31} is obvious, so we can assume that~$T\in (0,+\infty)$.
Consequently, we have~$\Psi'(T)=0$, and therefore
\[
0=\frac{\Psi'(T)}{T}
=1-T^{2^*_s-2}-c\mu^\beta T^{m-2}+C\mu^N T^{2^*_s-2}+C\mu^{N-2s}.
\]
We set
\[
\Phi_\mu(t):=
1-t^{2^*_s-2}-c\mu^\beta t^{m-2}+C\mu^N t^{2^*_s-2}+C\mu^{N-2s}
\]
and we have that~$T=T(\mu)$ is a solution of~$\Phi_\mu(t)=0$.

We remark that
\[
\Phi_\mu'(t)=
-(2^*_s-2)(1-C\mu^N)t^{2^*_s-3}-c\mu^\beta (m-2)t^{m-3}<0,
\]
since~$m-2\geq 0$ (recall~\eqref{eq6.5.22})
and~$(2^*_s-2)(1-C\mu^N)\geq 0$ if~$\mu$ is 
small enough.
This says that~$\Phi_\mu$ is strictly decreasing, hence~$T=T(\mu)$ is the unique solution of~$\Phi_\mu(t)=0$.

It is convenient to write~$\tau(\mu):=T(\mu^\frac{1}{\beta})$
and~$\eta:=\mu^\beta$, so that our equation becomes
\[
\begin{aligned}
0=\Phi_\mu(T(\mu))=\Phi_\mu(\tau(\mu^\beta))=\Phi_\mu(\tau(\eta))
=1-\big(1-C\eta^\frac{N}{\beta}\big)(\tau(\eta))^{2^*_s-2}
-c\eta (\tau(\eta))^{m-2}+C\eta^\frac{N-2s}{\beta}.
\end{aligned}
\]
Hence, if we differentiate in~$\eta$, we have that
\begin{equation}\label{eq6.5.32}
\begin{aligned}
0&=\frac{\partial}{\partial \eta}\left(1-(1-C\eta^\frac{N}{\beta})(\tau(\eta))^{2^*_s-2}
-c\eta (\tau(\eta))^{m-2}+C\eta^\frac{N-2s}{\beta} \right) \\
&=-(2^*_s-2)\big(1-C\eta^\frac{N}{\beta}\big)(\tau(\eta))^{2^*_s-3}
\tau'(\eta)
+C\frac{N}{\beta}\eta^{\frac{N}{\beta}-1}(\tau(\eta))^{2^*_s-2} \\
&\qquad-c (\tau(\eta))^{m-2}-c(m-2)\eta (\tau(\eta))^{m-3}\tau'(\eta)
+\frac{C(N-2s)}{\beta}\eta^{\frac{N-2s}{\beta}-1}.
\end{aligned}
\end{equation}

Now, we observe that, by~\eqref{eq6.5.28},
\[
\begin{aligned}
\frac{N-2s}{\beta}-1&=
\begin{cases}
\dfrac{2(N-2s)}{N}-1   &  \mbox{if } N>4s,
\\ \\
2-1              &  \mbox{if } N\in(2s,4s],
\end{cases}
\\
&=
\begin{cases}
\dfrac{N-4s}{N}   &  \mbox{if } N>4s,
\\ \\
1             &  \mbox{if } N\in(2s,4s],
\end{cases}
\end{aligned}
\]
and thus
\begin{equation}\label{eq6.5.33}
\frac{N-2s}{\beta}-1>0.
\end{equation}

We also notice that when~$\mu=0$, we have that~$T=1$ is a solution of~$\Psi_0(t)=0$, that is~$T(0)=1$, and so~$\tau(0)=1$.
Hence, we evaluate~\eqref{eq6.5.32} at~$\eta=0$ and use~\eqref{eq6.5.33} to see that~$0=-(2^*_s-2)\tau'(0)-c$.
In this way,
\[
\tau'(0)=-\frac{c}{2^*_s-2},
\]
which gives that
\[
\tau(\eta)=1-\frac{c\eta}{2^*_s-2}+o(\eta),
\]
and accordingly
\[
T(\mu)=\tau(\mu^\beta)
=1-\frac{c\mu^\beta}{2^*_s-2}+o(\mu^\beta)
=1-c_0\mu^\beta+o(\mu^\beta),
\]
for some~$c_0>0$. 

Consequently, we have that
\[
\begin{aligned}
\sup_{t\geq 0}\Psi(t)&=\Psi(T(\mu)) \\
&=(1+C\mu^{N-2s})\frac{(T(\mu))^2}{2}
-(1-C\mu^N)\frac{(T(\mu))^{2^*_s}}{2^*_s}
-\frac{c\mu^\beta (T(\mu))^m}{m} \\
&=(1+C\mu^{N-2s})\frac{(1-c_0\mu^\beta+o(\mu^\beta))^2}{2}
-(1-C\mu^N)\frac{(1-c_0\mu^\beta+o(\mu^\beta))^{2^*_s}}{2^*_s} \\
&\qquad\quad -\frac{c\mu^\beta (1-c_0\mu^\beta+o(\mu^\beta))^m}{m} \\
&=(1+C\mu^{N-2s})\frac{1-2c_0\mu^\beta}{2}
-(1-C\mu^N)\frac{1-2^*_sc_0\mu^\beta}{2^*_s}
-\frac{c\mu^\beta}{m}+o(\mu^\beta) \\
&=\frac{1-2c_0\mu^\beta}{2}-\frac{1-2^*_sc_0\mu^\beta}{2^*_s}
-\frac{c\mu^\beta}{m}+o(\mu^\beta) \\
&=\frac{1}{2}-\frac{1}{2^*_s}-\frac{c\mu^\beta}{m}+o(\mu^\beta)\\
&<\frac{1}{2}-\frac{1}{2^*_s}\\&=\frac{s}{N},
\end{aligned}
\]
which proves~\eqref{eq6.5.31}.

Using~\eqref{eq6.5.30} and~\eqref{eq6.5.31}, we obtain that
\[
\sup_{t\in[0,T_1]}I_\varepsilon^\star(t\phi z_{\mu,\xi})
\leq S^\frac{N}{2s}\sup_{t\geq 0}\Psi(t)
<\frac{s}{N} S^\frac{N}{2s},
\]
which gives~\eqref{eq6.5.19} and concludes the proof of Lemma~\ref{lemma6.5.7}. 
\end{proof}

As a consequence of~\eqref{feiwhwbnrgjkthioi8eu56895} and Lemma~\ref{lemma6.5.7}, we have that:

\begin{cor}\label{lemma6.5.7BIS}
There exists~$\mu_\star\in (0,\mu_0)$ such that if~$\mu\in (0,\mu_\star)$ then
\begin{equation*}
\sup_{t\geq 0}I_\varepsilon(t\phi z_{\mu,\xi})
<\frac{s}{N}S^\frac{N}{2s}.
\end{equation*} 
\end{cor}

We are now ready to prove the existence of a second solution to~\eqref{problema} and thus complete the proof of Theorem~\ref{thduesoluzioni}.

\begin{prop} \label{prop:wsecsol98}
Suppose that~$(h_0)$ and~$(h_1)$ hold true.

Then, problem~\eqref{problema} admits a second nonnegative solution~$\widetilde u_\varepsilon$ that does not vanish identically.

Moreover, $\widetilde u_\varepsilon\in L^\infty(\R^N)\cap C^\alpha(\R^N)$, for any~$\alpha\in(0,\min\{2s,1\})$.

Also, if~$h\ge0$ in some open set~$\Omega\subseteq\R^N$, then~$\widetilde u_\varepsilon(x)>0$ for all~$x\in\Omega$.

In particular,
if~$h\ge0$, then~$\widetilde u_\varepsilon(x)>0$ for all~$x\in\R^N$.
\end{prop}

\begin{proof}
In order to prove Proposition~\ref{prop:wsecsol98}, we look
at the translated functional~$I_\varepsilon$ in~\eqref{definizioneIepsilon}.
Thanks to Proposition~\ref{propvminimolocale}, we know that~$0$ is a critical
point of~$I_\varepsilon$. 

We now claim that
\begin{equation}\label{dhewugt3y4t734876lkjhgfd}
{\mbox{there exists a second critical point~$v$ of~$I_\varepsilon$ that does not vanish identically.}}
\end{equation}
To check this, we argue towards a contradiction and we suppose that~$0$
is the only critical point of~$I_\varepsilon$. In particular, we know that~$0$
is a local minimum for~$I_\varepsilon$. This and Corollary~\ref{lemma6.5.3} give that~$I_\varepsilon$ satisfies the geometric hypothesis of the Mountain
Pass Theorem. 

Moreover, we are in the position of employing Proposition~\ref{PSIepsilon}
to deduce that~$I_\varepsilon$ satisfies the Palais-Smale condition under the
level~$\frac{s}{N} S^\frac{N}{2s}$.

Also, denoting by~$c$ the mountain 
pass level of~$I_\varepsilon$, Corollary~\ref{lemma6.5.7BIS} gives that
\[
c\leq \sup_{t\geq 0}I_\varepsilon^\star(t\phi z_{\mu,\xi})
<\frac{s}{N} S^\frac{N}{2s}.
\]
Hence~$I_\varepsilon$ admits a critical point of mountain pass type,
which gives the desired contradiction. Thus, the claim in~\eqref{dhewugt3y4t734876lkjhgfd} is established.

In light of~\eqref{dhewugt3y4t734876lkjhgfd}
and Lemma~\ref{NUOVLEM}, we have that~$\widetilde u_\varepsilon:=u_\varepsilon+v$
is a nonnegative (and nontrivial) weak solution of~\eqref{problema}, as desired.

Furthermore, $\widetilde u_\varepsilon \in L^\infty(\R^N)\cap C^\alpha(\R^N)$, for any~$\alpha\in(0,\min\{2s,1\})$, thanks to
Proposition~5.1.1 and Corollary~5.1.3 in~\cite{MR3617721}.

Finally, we observe that if~$h\geq 0$ in~$\Omega$, then~$(-\Delta)^s \widetilde u_\varepsilon\ge0$ in~$\Omega$ in the weak sense, and therefore 
the strong maximum principle for continuous weak solutions in
Proposition~5.2.1 in~\cite{MR3617721} yields that the solution~$u_\varepsilon$ is strictly positive in~$\Omega$.
\end{proof}

\begin{appendix}

\section{On Proposition~3.1.1 in~\cite{MR3617721}}\label{APP}

As mentioned in the introduction, the results presented here were proved in~\cite{MR3617721}
in the parameter range~$s\in\left(0,\frac12\right]$. It was incorrectly stated in~\cite{MR3617721} that the methods presented there were able to capture the full range~$s\in(0,1)$ and the objective of this appendix is to clarify ``what goes wrong'' in~\cite{MR3617721} when~$s>\frac12$.

As a matter of fact, all the methods in~\cite{MR3617721} can cover the full fractional range~$s\in(0,1)$, with the exception of Proposition~3.1.1 there, which is only valid
when~$s\in\left(0,\frac12\right]$.

The technical issue in the proof of Proposition~3.1.1 in~\cite{MR3617721} is that it relied on Muckenhoupt weights, which is totally fine when~$s\in\left(0,\frac12\right]$,
but when~$s>\frac12$ the exponent~$q$ used in that proof becomes less than~$1$, thus going away from the Muckenhoupt class. At a first glance, this may seem a minor technical glitch, to be fixed through one of the many weighted Sobolev inequalities that are available in the literature -- and one may suspect that this is the case also because usually the parameter range below~$\frac12$ (not above!) is the more ``problematic'' one, and one may imagine that for inequalities valid for~$s=1$ and below~$s=\frac12$ the use of a smart interpolation trick could fix all technical issues.

But no, the situation is surprisingly more complicated, as shown by the next observation.

\begin{lem}
Proposition~3.1.1 in~\cite{MR3617721} is not valid
when~$s\in\left(\frac12,1\right)$.

More specifically, for all~$N\ge2$ and~$s\in\left(\frac12,1\right)$, setting~$\gamma:=1+\frac{2}{N-2s}$, there exists
a sequence of functions~$U_n\in C^\infty_c(\R^n\times(0,+\infty),\,[0,1])$ such that
$$ \lim_{n\to+\infty}\frac{\displaystyle\left(
\int_{\R^N\times(0,+\infty)}y^{1-2s} |U_n(x,y)|^{2\gamma}\,dx\,dy\right)^{\frac1{2\gamma}}}{\displaystyle\left(\int_{\R^N\times(0,+\infty)}y^{1-2s} |\nabla U_n(x,y)|^{2}\,dx\,dy\right)^{\frac12}}=+\infty.$$
\end{lem}

\begin{proof} We use the short notation~$\R^{N+1}_+:=\R^N\times(0,+\infty)$
and denote points by~$X=(x,y)\in\R^n\times(0,+\infty)$.

Let~$R>1$ and~$\Phi\in C^\infty_c(B_1,[0,1])$ with~$\Phi=1$ in~$B_{1/2}$.
Let also~$U(X)=U(x,y):=\Phi(x, y-R)$.

Then, if~$s\in\left(\frac12,1\right)$,
\begin{eqnarray*}&&
\int_{\R^{n+1}_+}y^{1-2s} |U(X)|^{2\gamma}\,dX\ge
\int_{B_{1/2}(Re_{n+1})}y^{1-2s} |\Phi(x, y-R)|^{2\gamma}\,dX\\&&\qquad=\int_{B_{1/2}(Re_{n+1})}y^{1-2s}\,dX\ge
|B_{1/2}|\,\left(R+\frac12\right)^{1-2s}
\end{eqnarray*}
and
\begin{eqnarray*}&&\int_{\R^{n+1}_+}y^{1-2s} |\nabla U(X)|^{2}\,dX
=\int_{B_1(Re_{n+1})}y^{1-2s} |\nabla \Phi(x,y-R)|^{2}\,dX\\&&\qquad
\le \left(R-\frac12\right)^{1-2s}\int_{B_1}|\nabla \Phi(X)|^{2}\,dX.
\end{eqnarray*}
As a result,
\begin{eqnarray*}&&\frac{\displaystyle\left(
\int_{\R^{n+1}_+}y^{1-2s} |U(X)|^{2\gamma}\,dX\right)^{\frac1{2\gamma}}}{\displaystyle\left(\int_{\R^{n+1}_+}y^{1-2s} |\nabla U(X)|^{2}\,dX\right)^{\frac12}}\ge c\,\frac{\displaystyle
\left(R+\frac12\right)^{\frac{1-2s}{2\gamma}}}{\displaystyle\left(R-\frac12\right)^{\frac{1-2s}2}}\ge cR^{\frac{1-2s}2\left(\frac1\gamma-1\right)}
\end{eqnarray*}
for some~$c>0$ possibly varying from step to step. 

Hence, since~$\gamma>1$, and thus~$\frac{1-2s}2\left(\frac1\gamma-1\right)>0$,
we conclude that$$ \lim_{R\to+\infty}\frac{\displaystyle\left(
\int_{\R^{n+1}_+}y^{1-2s} |U(X)|^{2\gamma}\,dX\right)^{\frac1{2\gamma}}}{\displaystyle\left(\int_{\R^{n+1}_+}y^{1-2s} |\nabla U(X)|^{2}\,dX\right)^{\frac12}}=+\infty,$$
as claimed.
\end{proof}

\end{appendix}

\vfill
\end{document}